\documentclass{math}

\usepackage{txfonts}
\usepackage{graphicx}
\usepackage{mathrsfs}
\def\typeofarticle{Research article}
\def\currentvolume{11}
\def\currentissue{9}
\def\currentyear{2026}

\def\ppages{xxx--xxx.}
\def\DOI{ 10.3934/math.2026xxx}
\def\Received{21 June 2026}
\def\Revised{04 August 2026}
\def\Accepted{25 August 2026}
\def\Published{September 2026}

\numberwithin{equation}{section}

\makeatletter
\renewcommand{\@biblabel}[1]{#1\hfill \hspace{-0.2cm}}
\makeatother

\usepackage{cite}
\usepackage{enumitem}

\newtheorem{thm}{Theorem}[section]
\newtheorem{cor}{Corollary}[section]
\newtheorem{lem}{Lemma}[section]

\newtheorem{prop}{Proposition}[section]

\theoremstyle{definition}
\newtheorem{defn}{Definition}[section]
\newtheorem{rem}{Remark}[section]
\newtheorem{ex}{Example}[section]

\DeclareMathOperator*{\esssup}{ess\,sup}
\begin{document}
	
\title{A functional-analytic framework for nonlinear adaptive memory: hierarchical kernels, state-dependent sensitivity, and memory-dependent functionals}

\author{%
	Jiahao Jiang\corrauth
}

\shortauthors{the Author(s)}

\address{School of Mathematics, Southwest Jiaotong University, Chengdu 610031, China}
\corraddr{Email: jiahao.jiang@swjtu.edu.cn.
}

\maketitle

\begin{abstract}
In this work, we developed a systematic framework for nonlinear adaptive memory within nonlinear functional analysis, where the influence of past events depends on both elapsed time and the state values along a trajectory. The framework comprises three hierarchical layers. First, memory kernels were classified into mathematically admissible, regular (uniformly bounded, normalized, Lipschitz), and generalized (bounded variation, possibly sign-changing) classes. Second, adaptive sensitivity functions $\Lambda(s,f(s))$ were introduced, satisfying natural conditions; a concrete construction based on historical deviation accumulation interpolated continuously between instantaneous response and history-dependent sensitivity, with an explicit Lipschitz estimate $\|\Lambda_f-\Lambda_g\|_\infty\le L_\Lambda\|f-g\|_\infty$. Third, an adaptive memory-dependent functional
\[
S_{\kappa,\Lambda}(f)
=
\sup_{t\in I}
\left(
|f(t)|
+
\int_0^t \Lambda(s,f(s))\,\kappa(t-s)\,|f(s)|\,ds
\right)
\]
and the associated set $\mathscr{M}_{\kappa,\Lambda}(I)=\left\{ f : S_{\kappa,\Lambda}(f)<\infty \right\}$ were constructed. Fundamental properties of the framework were established, including absolute convergence, measurability, uniform boundedness, positive definiteness, and comparison with the classical supremum norm. It was shown that $\mathcal{C}(I) \subset \mathscr{M}_{\kappa,\Lambda}(I)$ strictly, with discontinuous functions (e.g., indicator functions of subintervals) belonging to the set, capturing abrupt signal changes such as on-off switching in nonlinear systems. When the maximum of $|f|$ was attained in $(0,T]$, a strict inequality $S_{\kappa,\Lambda}(f)>\|f\|_\infty$ was proved, demonstrating the nontrivial contribution of the memory component. The resulting framework could be viewed as a mathematical description of nonlinear adaptive phenomena, including habituation, state-dependent weighting, and selective retention, within a rigorous functional-analytic setting. By separating temporal weighting from state-dependent modulation, the construction offered a modular methodology applicable to adaptive phenomena in fields such as neuroscience, adaptive control, and machine learning, where memory is time-dependent and content-sensitive, and may be of interest across other mathematical disciplines, as a step toward a broader understanding of nonlinear phenomena from a mathematical perspective.
\end{abstract}
	
\section{Introduction}
The influence of past events is not just about time; it also depends on what happened. A strong event may leave a different trace than a weak one, even if both happened equally long ago. This paper attempts to develop a mathematical framework for such adaptive memory, where past events are weighted not only by when they occurred but also by their size. The framework introduces a new way to measure signals that combines their present strength with their adaptively weighted history. This work may offer a theoretical foundation for understanding habituation, selective retention, and other memory-based adaptation phenomena in nonlinear science, from the perspective of mathematical methods for nonlinear systems, and could be of relevance to fields such as neuroscience, control theory, and machine learning.

Memory, the persistence of past events in shaping present behavior, is a pervasive feature across scientific disciplines. In physical systems, heat conduction with memory effects modifies the classical diffusion paradigm, leading to anomalous large-time asymptotics that depend critically on the space-time scale under consideration \cite{cortazar2021heat,cortazar2024asymptotic}. In material science, nonclassical diffusion equations with memory terms on time-dependent spaces exhibit complex dynamic behavior, including the existence of global attractors under general assumptions on the memory kernel \cite{wangattractor}. In structural mechanics, Timoshenko systems incorporating memory terms display a dissipative structure whose spectral analysis yields optimal decay estimates and global existence results in critical Sobolev spaces \cite{mori2018dissipative}. Extensions to thermoelastic systems with time-varying delays further reveal global existence, asymptotic behavior, and uniform attractors \cite{qin2020global}. In stochastic analysis, semilinear Volterra equations with multiplicative noise and non-analytic kernels exhibit parabolic character, with well-posedness and regularity established under Lipschitz-type conditions \cite{baeumer2015existence}. A related class of nonlinear Volterra-Fredholm integro-differential equations has been addressed using evolutionary computational intelligence techniques, where artificial neural networks provide accurate modeling of the integral structure \cite{kashkaria2017evolutionary}. In optimal control, a posteriori error estimates for finite-element discretizations of nonlinear parabolic integro-differential control problems enable adaptive multi-mesh schemes \cite{lu2014posteriori}.

Beyond deterministic and stochastic evolution equations, memory concepts appear prominently in neural network theory. Neutral-type hybrid bidirectional associative memory networks with time-varying delays have been analyzed for robust stability, yielding delay-derivative-dependent conditions for equilibrium \cite{feng2014robust}. Hybrid fractional differential equations combined with artificial neural networks have been employed to model tuberculosis transmission dynamics, demonstrating the practical utility of memory-based modeling in epidemiology \cite{manu2025mathematical}. In time series analysis, spectral methods for multifractional long-range dependent functional time series characterize memory through operator-valued spectral densities, with weak-consistent estimation of long-memory operators under Gaussian assumptions \cite{ruiz2022spectral}. In social dynamics, memory-based reduced modeling of opinion spreading, grounded in the Mori-Zwanzig projection formalism, shows that inclusion of memory terms significantly improves prediction quality across network topologies \cite{wulkow2021memory}. In a related direction, further inequalities for the Wills functional of convex bodies have been established, including Brunn-Minkowski and Rogers-Shephard-type inequalities, as well as a new proof of McMullen's inequality~\cite{alonso2021further}. In fluid mechanics, incompressible flows achieve exponential self-similar mixing of passive tracers, a phenomenon intimately connected to the decay of memory in transported quantities~\cite{alberti2019exponential}. In geometry, norms on the cohomology of hyperbolic 3-manifolds relate purely topological Thurston norms to more geometric harmonic norms, with explicit constants depending on volume and injectivity radius \cite{brock2017norms}.

Foundational to many of these analyses is the decomposition of functions or operators into components of definite sign or monotonicity. Classical Jordan decomposition theorems for signed measures, functions of bounded variation, and operators in Banach spaces provide essential tools for handling memory kernels that may change sign or exhibit irregular behavior \cite{schmidt1982general,leonov1996total}. Further decomposition results and their extensions to maps of finite total variation and to operators in Banach spaces are established in \cite{chistyakov2010maps,ene1998decomposition,kantorovitz1965jordan}. In parallel, inequalities for continuous Archimedean t-norms generalize classical Minkowski-type estimates and offer a functional framework relevant to the analysis of nonlinear integral operators \cite{saminger2008generalization}. These ideas extend to fuzzy settings, where interval-valued fuzzy structures with t-norms provide computational tools for defuzzification in fuzzy control \cite{davvaz2009applications}. In approximation theory, artificial neural networks with uniform-norm-based loss functions offer alternative training paradigms when data are limited or class sizes are imbalanced \cite{peiris2024artificial}.

Adaptive algorithms have been extensively developed for solving inverse problems and fixed-point problems. Moreover, Adaptive Gauss-Newton methods for nonlinear equations use residual bounds and quadratic regularization to achieve global convergence without prior knowledge of hyperparameters~\cite{yudin2021adaptive}. Split common fixed-point problems for demicontractive operators have been solved using adaptive algorithms that do not require operator norms, with strong convergence established under suitable conditions \cite{kitkuan2019adaptive}. Inertial self-adaptive algorithms with two different inertial factors approximate minimum-norm solutions of split feasibility problems in Banach spaces, employing step-size selection without prior norm knowledge \cite{sunthrayuth2025inertial}. Multiple-sets split equality problems have been addressed with iterative algorithms whose split self-adaptive step sizes are computed directly from the iterative procedure \cite{tian2016iterative}. Split fixed-point problems for multi-valued demi-contractive mappings admit self-adaptive algorithms with strong convergence guarantees \cite{jailoka2021split}. Limited-memory bundle methods for difference-of-convex optimization enable sparse pairwise kernel learning with non-smooth loss functions, balancing prediction accuracy with sparsity requirements~\cite{karmitsa2025limited}. Maximum-norm a posteriori error bounds for extrapolated upwind schemes applied to singularly perturbed convection-diffusion problems provide robust estimators that guide adaptive mesh generation \cite{linss2024maximum}. Memory-efficient combinatorial attacks on small LWE keys address the challenge of limited memory in cryptographic contexts, outperforming previous approaches when memory is constrained \cite{esser2024memory}. Physics-based active learning strategies iteratively sample training points based on physical query rules, constructing surrogate models that are accurate in industrially relevant regions of the parameter space \cite{torregrosa2024physics}. Uniform convergence of multigrid methods for elliptic quasi-variational inequalities has been established in the maximum norm, with applications to impulse control problems \cite{belouafi2023uniform}. Weak-duality-based adaptive finite element methods for PDE-constrained optimization with pointwise gradient state constraints derive residual-based a posteriori error estimators without requiring constraint qualifications \cite{hintermuller2012weak}. Contextual inverse multiobjective optimization recovers multiple objective functions and preferences from observed context-decision pairs, accommodating different norm-based scalarizations \cite{blanquero2025contextual}. Inverse moment bounds for sample autocovariance matrices based on detrended time series provide mean-square error bounds for finite predictor coefficients under short- or long-memory dependence~\cite{cheng2015inverse}. Probabilistic adaptive widths of multivariate Sobolev spaces equipped with Gaussian measures determine asymptotic values for approximation in $L^q$ norms~\cite{guanggui2006probabilistic}. Extensions of radial basis functions address large-scale learning problems where different variables play distinct roles, with applications in artificial intelligence \cite{girosi1992some}. Bio-inspired neural network architectures modeled on the cerebellum offer avenues for embodied intelligence, where adaptive neural networks learn through physical interaction \cite{nurutdinov2024bio}. Theories of attribute implications in multi-adjoint concept lattices with hedges provide logical foundations for knowledge representation and decision-support systems \cite{cornejo2026theories}. The relationship between large deviation rate functions and Kullback-Leibler divergence informs the interpretation of neural estimation of mutual information \cite{tsuda2026some}. Integration of privacy-enhancing technologies into explainable artificial intelligence mitigates attribute inference attacks on feature-based explanations, reducing attack success while preserving model utility~\cite{allana2025towards}. Computationally enhanced projection methods for symmetric Sylvester and Lyapunov matrix equations compute residual norms at reduced cost, making Krylov strategies competitive with more recent approaches \cite{palitta2018computationally}. Perturbation analysis of singular values in concatenated matrices extends classical Weyl inequalities, providing stability bounds for low-rank approximations with applications in signal processing and data-driven modeling \cite{shamrai2025analysis}. Generalizations of fractional calculus have been systematically investigated to study anomalous stochastic processes, leading to fractional Itô calculus and generalized Fokker-Planck equations that describe underdamped and overdamped stochastic dynamics \cite{jiang2025study}. Adaptive frequency evolution decomposition combined with improved fluctuation-based dispersion entropy has been proposed for muscle fatigue characterization using surface electromyography signals, demonstrating the utility of adaptive signal decomposition in biomedical applications \cite{dong2026adaptive}. A mathematical theory of adaptive memory has been developed for stochastic processes whose local regularity adapts dynamically in response to their own state, introducing time-varying and responsive fractional Brownian motion with rigorous analysis of covariance structure, pathwise regularity, and attention-like mechanisms \cite{jiang2025towards}. A two-parameter memory-weighted velocity operator has been introduced and analyzed, providing a framework for describing rates of change in systems with time-varying, power-law memory, with weighted pointwise estimates revealing compensation mechanisms between independent memory weightings \cite{jiang2026foundations}. In a related direction, a semiclassical functional calculus on nilpotent Lie groups and their compact nilmanifolds has been developed, showing that this calculus includes the functional calculus of suitable subelliptic operators; moreover, Weyl laws for these operators are obtained, with examples including sub-Laplacians in horizontal divergence form perturbed by a potential and their generalizations \cite{fischer2025semiclassical}. Furthermore, asymptotic weighted approximation of convex functions by piecewise affine convex functions has been studied, where the asymptotic expressions are formulated in terms of a functional that extends the notion of affine surface area to the functional setting as the number of facets tends to infinity \cite{baeta2026asymptotic}.

The above-mentioned works provide important mathematical descriptions of memory phenomena in their respective application contexts, and a number of them adopt fixed kernels based on time lags. Furthermore, in real-world systems, such as neural adaptation, gain scheduling, and content-addressable memory, the influence of a past event also depends on its magnitude or novelty at the time it occurred. This suggests that, in addition to time dependence, state dependence may be a relevant factor in the design of memory weights. In this context, we attempt to develop a systematic functional-analytic framework in which the memory weight depends not only on the time lag, but also on the instantaneous state of the trajectory through an adaptive sensitivity function, thereby achieving a twofold dependence on both time and state. This design enables the memory mechanism to reflect more flexibly the characteristics of the signal, and may offer a possible mathematical foundation for describing nonlinear adaptive phenomena such as habituation and selective retention. This twofold dependence on time and state constitutes the point of departure for the present framework, which is organized into three hierarchical layers and will be developed in detail in the following sections.

In this work, we build upon this broad landscape of memory-related models and adaptive algorithms. The framework is developed through three successive layers: A hierarchical classification of memory kernels, a class of adaptive sensitivity functions that modulate the kernel weights according to the state of the trajectory, and a memory-dependent functional with its associated function set that integrates instantaneous magnitude and adaptively weighted historical accumulation. First, a classification of memory kernels is developed in Section~\ref{sec:kernel_theory}: Starting from mathematically admissible kernels satisfying minimal integrability and measurability conditions (Definition~\ref{def:math_admissible_kernel}), regular admissible kernels are introduced with uniform boundedness, normalization, and Lipschitz continuity (Definition~\ref{def:regular_admissible_kernel}); generalized admissible kernels further relax these requirements to allow sign changes, bounded variation, and non-unit total weight (Definition~\ref{def:general_admissible_kernel}). Stationary memory kernels of the form $K_\kappa(t,s)=\kappa(t-s)$ are adopted as a foundational premise (Subsection~\ref{subsec:stationary_kernel}), and prototypical examples, exponential kernel (Example~\ref{ex:exponential_kernel}), power-law kernel (Example~\ref{ex:power_law_kernel_generalized}), and finite-memory kernel (Example~\ref{ex:finite_memory_kernel}), illustrate the scope of the classification. Essential integral estimates for regular and generalized kernels are established in Lemma~\ref{lem:kernel_integral_estimates} and Proposition~\ref{prop:generalized_estimates}. 

Second, adaptive sensitivity functions
\[
\Lambda \colon I \times \mathbb{R} \to [0,\infty)
\]
are introduced in Section~\ref{sec:adaptive_sensitivity_framework} (Definition~\ref{def:adaptive_sensitivity}), satisfying the conditions of uniform boundedness, Lipschitz continuity in the state variable, measurability in time, and positivity at the zero state. These functions modulate the memory weight according to the state value $f(s)$ along the trajectory, yielding a composite weight $\Lambda(s,f(s))\,\kappa(t-s)$. A concrete construction based on historical deviation accumulation (Example~\ref{ex:historical_deviation_adaptive}) demonstrates that the framework accommodates genuinely history-dependent sensitivity, interpolating continuously between purely instantaneous response ($\beta_0=0$) and operator-level historical feedback ($\beta_0>0$). Theorem~\ref{thm:lambda_f_properties} establishes fundamental properties of the resulting function $\Lambda_f$: uniform boundedness (P1), a Lipschitz-type estimate in the supremum norm $\|\Lambda_f-\Lambda_g\|_\infty\le L_\Lambda\|f-g\|_\infty$ (P2), continuity in time (P3), and positivity at zero (P4). Corollary~\ref{cor:beta_zero_case} verifies that the purely instantaneous case reduces to a function $\Lambda\in\mathscr{A}(I)$.

Third, the adaptive memory-dependent functional $S_{\kappa,\Lambda}(f)$ is constructed in Section~\ref{sec:adaptive_memory_sets} via a supremum operation (Definition~\ref{def:adaptive_memory_functional}), integrating the instantaneous magnitude $|f(t)|$ with the adaptively weighted historical accumulation $\mathcal{J}_f(t)$ (Definition~\ref{def:adaptive_basic_functions}). Lemma~\ref{lem:functions_well_posed} establishes absolute convergence, measurability, and uniform boundedness of the auxiliary functions $\mathcal{J}_f$ and $\mathcal{M}_f$. Lemma~\ref{lem:adaptive_functional_properties} verifies existence, controlled boundedness, positive definiteness, and comparison with the classical supremum norm for $S_{\kappa,\Lambda}(f)$. Theorem~\ref{thm:strict_functional_comparison} shows that when the maximum of $|f|$ is attained in $(0,T]$, the inequality is strict: $S_{\kappa,\Lambda}(f) > \|f\|_\infty$. The associated function set
\[
\mathscr{M}_{\kappa,\Lambda}(I) = \{ f : S_{\kappa,\Lambda}(f) < \infty \}
\]
is introduced in Definition~\ref{def:adaptive_memory_set}. Theorem~\ref{thm:embedding_C_into_M} proves that $\mathcal{C}(I) \subset \mathscr{M}_{\kappa,\Lambda}(I)$ and establishes the two-sided estimate
\[
\|f\|_\infty \le S_{\kappa,\Lambda}(f) \le \bigl(1 + \Lambda_\infty\,\kappa_\infty\,T\bigr) \|f\|_\infty.
\]
Proposition~\ref{prop:inclusion_estimate} shows that the inclusion mapping is linear, satisfies a controlled estimate, and is Lipschitz continuous. Proposition~\ref{prop:discontinuous_bounded_inclusion} demonstrates that the set also contains certain discontinuous functions, such as the indicator function of a subinterval, reflecting the framework's flexibility to accommodate signals with jump discontinuities arising in nonlinear phenomena like abrupt environmental changes or on-off switching.

The structure of the paper is as follows: In Section~\ref{sec:kernel_theory}, we develop the hierarchical classification of memory kernels, introducing mathematically admissible, regular, and generalized kernels, together with stationary representations, prototypical examples, and essential integral estimates. In Section~\ref{sec:adaptive_sensitivity_framework}, we introduce adaptive sensitivity functions, establish their axiomatic properties, and present a concrete construction based on historical deviation accumulation, including detailed verification of uniform boundedness, Lipschitz-type estimates, continuity, and positivity at zero. In Section~\ref{sec:adaptive_memory_sets}, we construct the adaptive memory-dependent functional $S_{\kappa,\Lambda}(f)$ and the associated function set $\mathscr{M}_{\kappa,\Lambda}(I)$, proving well-posedness, fundamental properties, and embedding results, including strict comparison when the maximizer lies in $(0,T]$.

\section*{Main notation and symbols}

For the reader's convenience, we list below the main notation used throughout the paper. All symbols are defined in context at their first occurrence. Table~\ref{tab:notation} summarizes the main notation and symbols used throughout the paper.

\begin{table}[htbp]
\centering
\renewcommand\arraystretch{1.1}
\setlength{\tabcolsep}{7mm}
\caption{Main notation and symbols used throughout the paper.}
\label{tab:notation}
\begin{tabular}{l p{8cm} p{3cm}}
\hline
Symbol & Description & First defined in \\ \hline
\(I\) & fixed time interval \([0,T]\) & Section~\ref{subsec:math_admissible_kernel} \\
\(\Omega\) & triangular region \(\{(t,s):0\le s\le t\le T\}\) & Section~\ref{subsec:math_admissible_kernel} \\
\(\mathscr{K}_{\mathrm{math}}\) & class of mathematically admissible kernels & Definition~\ref{def:math_admissible_kernel} \\
\(\mathscr{K}_{\mathrm{reg}}\) & class of regular admissible kernels & Definition~\ref{def:regular_admissible_kernel} \\
\(\mathscr{K}_{\mathrm{gen}}\) & class of generalized admissible kernels & Definition~\ref{def:general_admissible_kernel} \\
\(\Lambda(s,x)\) & adaptive sensitivity function & Definition~\ref{def:adaptive_sensitivity} \\
\(\Lambda_f(s)\) & trajectory-dependent sensitivity function & Example~\ref{ex:historical_deviation_adaptive} \\
\(\mathcal{J}_f(t)\) & adaptive weighted cumulative function & Definition~\ref{def:adaptive_basic_functions} \\
\(\mathcal{M}_f(t)\) & instantaneous-memory hybrid function & Definition~\ref{def:adaptive_basic_functions} \\
\(S_{\kappa,\Lambda}(f)\) & adaptive memory-dependent functional & Definition~\ref{def:adaptive_memory_functional} \\
\(\mathscr{M}_{\kappa,\Lambda}(I)\) & adaptive memory-dependent sensitivity set & Definition~\ref{def:adaptive_memory_set} \\
\hline
\end{tabular}
\end{table}

\section{A hierarchical framework for memory kernels: concepts and estimates}\label{sec:kernel_theory}

We begin with the first layer of the framework: the classification of memory kernels. This part establishes the basic objects on which the subsequent layers will be built. Our purpose of this section is to introduce a self-consistent axiomatic framework for describing how past states are weighted in their influence on the present. The central object is the memory kernel, a function that quantifies the relative importance of a historical instant $s$ when examining the system at a later time $t$. Kernels of this type appear across a broad spectrum of mathematical models: in heat equations with memory \cite{cortazar2021heat, cortazar2024asymptotic}, Timoshenko systems of memory type \cite{mori2018dissipative}, nonclassical diffusion equations with memory~\cite{wangattractor}, Volterra equations and Volterra-Fredholm integro-differential equations \cite{baeumer2015existence, kashkaria2017evolutionary}, a posteriori error estimates for finite-element discretizations of parabolic integro-differential optimal control problems \cite{lu2014posteriori}, neural networks with time-varying delays \cite{feng2014robust}, hybrid fractional differential equations \cite{manu2025mathematical}, long-range dependent time series \cite{ruiz2022spectral}, the development of a mathematical theory for adaptive memory \cite{jiang2025towards}, and the study of a two-parameter memory-weighted velocity operator \cite{jiang2026foundations}.

A key challenge that emerges from these applications is that memory is rarely a static, time-invariant weighting of the past; instead, it often exhibits intrinsically nonlinear characteristics. Systems may adapt their responsiveness based on the intensity or frequency of past stimuli (a phenomenon known as habituation), or they may assign different weights to past events depending on the state values encountered along the trajectory; what we term state-dependent weighting. Our axiomatic framework addresses this class of phenomena by constructing a hierarchical classification of kernels that, in Sections~\ref{sec:adaptive_sensitivity_framework} and~\ref{sec:adaptive_memory_sets}, will be combined with state-dependent sensitivity functions to yield a theory of adaptive memory that accommodates such nonlinear features. This framework aims to capture the essential mathematical structures shared by the models cited above while remaining self-contained. In doing so, it provides a unified language for describing memory phenomena that are inherently nonlinear, adaptation, habituation, state-dependent weighting, within a rigorous functional-analytic~setting.

We proceed by first imposing basic mathematical requirements that guarantee the well-definedness of the relevant expressions (see Definition~\ref{def:math_admissible_kernel}); subsequently, to provide the necessary analytical tools for the development of the functional-analytic theory that follows, in particular, for the construction of adaptive sensitivity functions in Section~\ref{sec:adaptive_sensitivity_framework} (see Definition~\ref{def:adaptive_sensitivity}) and adaptive memory sets in Section~\ref{sec:adaptive_memory_sets} (see Definitions~\ref{def:adaptive_memory_functional} and~\ref{def:adaptive_memory_set}), we introduce additional regularity conditions (see Definition~\ref{def:regular_admissible_kernel}) that will be employed throughout the main body of this work.

\subsection{Mathematically admissible kernels: minimal requirements and basic properties}\label{subsec:math_admissible_kernel}

Throughout this work, we fix a finite time horizon $T \!>\! 0$ and denote by $I \! \coloneqq \! [0,\, T]$ the corresponding compact time interval, which will be used throughout without further mention. Memory effects, understood as the dependence of a system's present state on its past history, will be described by functions defined on the triangular region
\begin{equation}\label{eq:triangular_region}
\Omega:=\{(t,s)\in I\times I \mid 0\le s\le t\},
\end{equation}
which reflects the fundamental causal principle that the past may influence the present, but not vice~versa.

\begin{defn}[Mathematically admissible kernel]\label{def:math_admissible_kernel}
A function $\kappa \colon I \! \to \! [0,\infty)$ is said to be a mathematically admissible kernel if it fulfills the following three elementary requirements:
\begin{description}[leftmargin=*, nosep, font=\normalfont]
\item[(M1)] \textbf{Non-negativity:} $\kappa(\tau)\ge 0$ for every $\tau\in I$;
\item[(M2)] \textbf{Integrability:} $\kappa\in L^1(I)$, i.e.
\[
\|\kappa\|_{L^1(I)} \coloneqq \int_0^T \kappa(\tau)\,d\tau < \infty;
\]
\item[(M3)] \textbf{Measurability:} $\kappa$ is Lebesgue measurable on $I$.
\end{description}

The collection of all mathematically admissible kernels will be denoted by $\mathscr{K}_{\mathrm{math}}$.
\end{defn}

\begin{rem}
Conditions (M1)--(M3) constitute the basic hypotheses required to ensure that all integral expressions involving $\kappa$ appearing in this work are mathematically well-defined. Non-negativity guarantees that $\kappa$ admits an interpretation as a weight, while integrability and measurability are standard prerequisites for the Lebesgue integration theory. At this stage, no further regularity, such as continuity or boundedness, is imposed; these will be added later as the theory develops and stronger properties become necessary. From a modeling perspective, non-negativity encodes the natural requirement that past events should not exert negative influence (i.e., no ``inverse memory"), while integrability ensures that the total weight of past experience remains finite, a natural condition for any physically plausible memory mechanism.
\end{rem}

\begin{rem}[Stationarity as a foundational premise]\label{rmk:stationarity_premise}
For the purposes of this work, we choose to take stationary kernels as the starting point for the theoretical development. Specifically, for any $(t,s)\in\Omega$, we write
\begin{equation}\label{eq:stationary_form}
K_\kappa(t,s):=\kappa(t-s),
\end{equation}
which expresses that the weight attributed to a historical state at time $s$ is determined exclusively by the elapsed time $\tau:=t-s\in[0,T]$. This stationarity assumption is adopted as a foundational premise for the theory; it captures the essential feature that the influence of past events evolves with the passage of time, while providing a clear structure for the subsequent analysis. A formal definition of stationary memory kernels will be given in subsection~\ref{subsec:stationary_kernel}, where their role in the overall framework is further elaborated.
\end{rem}

\subsection{Regular admissible kernels: enhanced regularity conditions}\label{subsec:regular_admissible_kernel}

The basic conditions (M1)--(M3) introduced in Definition~\ref{def:math_admissible_kernel} guarantee that the integral expressions involving the kernel are mathematically well-defined. For the development of a deeper functional-analytic theory, and to model memory mechanisms where the weighting of past events varies in a sufficiently regular, non-erratic manner, stronger regularity properties will be imposed. These additional conditions ensure that the kernel exhibits the kind of smooth, bounded behavior expected in many natural memory processes, from biological adaptation to material relaxation, while providing the analytical tools needed for a rigorous functional-analytic treatment. In this subsection, we therefore introduce additional requirements that will be in force throughout the remainder of this~work.

\begin{defn}[Regular admissible kernel]\label{def:regular_admissible_kernel}
A mathematically admissible kernel $\kappa\in\mathscr{K}_{\mathrm{math}}$ is called a regular admissible kernel if it satisfies the following three enhanced conditions:
\begin{description}[leftmargin=*, nosep, font=\normalfont]
\item[(R1)] \textbf{Uniform boundedness:} There exist constants $0<m_\kappa\le M_\kappa<\infty$ such that
\[
m_\kappa\le\kappa(\tau)\le M_\kappa\qquad\text{for all }\tau\in I;
\]
\item[(R2)] \textbf{Normalization:}
\[
\int_0^T\kappa(\tau)\,d\tau=1;
\]
\item[(R3)] \textbf{Lipschitz continuity:} There exists a constant $L_\kappa>0$ such that
\[
|\kappa(\tau_1)-\kappa(\tau_2)|\le L_\kappa|\tau_1-\tau_2|\qquad\text{for all }\tau_1,\tau_2\in I.
\]
\end{description}

The collection of all regular admissible kernels will be denoted by $\mathscr{K}_{\mathrm{reg}}$.
\end{defn}

\begin{rem}[Mathematical and physical interpretation of the regularity conditions]
The regularity conditions (R1)--(R3) carry multiple layers of significance, from a mathematical and from a modeling perspective.
\begin{description}[leftmargin=*, nosep, font=\normalfont]
\item[(R1)] Uniform boundedness excludes pathological situations where the kernel would grow without bound or become arbitrarily small, thereby keeping the memory weights within a controllable range and avoiding mathematical singularities.
\item[(R2)] Normalization endows $\kappa$ with a statistical interpretation as a probability density function on the time-lag interval $[0,T]$. Consequently, the integral $\int_0^t \kappa(t-s)\,ds$ can be viewed as the ``cumulative memory intensity'' up to time $t$. Normalization also makes the total memory strength comparable across different kernels.
\item[(R3)] Lipschitz continuity provides sufficient smoothness to ensure that the kernel evolves gradually with the time lag. This not only facilitates subsequent differential calculus involving $\kappa$, but also aligns with the physical intuition that the weighting of past events evolves in a continuous manner.
\end{description}

Together, these conditions single out $\mathscr{K}_{\mathrm{reg}}$ as a class of kernels that not only possess favorable mathematical properties but also capture the essential qualitative features of many empirically observed memory processes: a finite total memory budget (normalization), a controllable range of influence (uniform boundedness), and a gradual, non-abrupt change in weighting with the passage of time (Lipschitz continuity). The uniform lower bound condition \(m_\kappa > 0\) in (R1) ensures the strict positivity of integrals involving \(\kappa\) on subintervals of positive length, providing a solid analytical foundation for the regular kernel class. Within this setting, the generalized kernel class \(\mathscr{K}_{\mathrm{gen}}\) introduced in Definition~\ref{def:general_admissible_kernel} further relaxes this restriction, enabling kernels that vanish at endpoints or decay asymptotically, thereby covering a broader range of physically and biologically relevant memory profiles.
\end{rem}

\subsection{Generalized admissible kernels: relaxed regularity conditions}\label{subsec:generalized_kernels}

The class $\mathscr{K}_{\mathrm{reg}}$ introduced in Subsection~\ref{subsec:regular_admissible_kernel} possesses excellent analytical properties and provides a solid foundation for the initial development of our theoretical framework. To further enhance the flexibility and applicability of our framework, we now introduce a broader class of kernels with weaker regularity requirements, thereby accommodating memory phenomena that may involve sign changes, jump discontinuities, or non-unit total weight, features that arise naturally in contexts such as inhibitory neural feedback, abrupt system resets, or adaptive gain control.

\begin{defn}[Generalized admissible kernel]\label{def:general_admissible_kernel}
A function $\kappa:I\to\mathbb{R}$ is called a generalized admissible kernel if it satisfies the following three conditions:
\begin{description}[leftmargin=*, nosep, font=\normalfont]
\item[(G1)] \textbf{Essential boundedness:} $\kappa\in L^\infty(I)$, i.e., there exists a constant $C_\kappa>0$ such that
\[
\esssup_{\tau\in I}|\kappa(\tau)|\le C_\kappa,
\]
where $\esssup$ denotes the essential supremum. This condition enables the kernel to exceed the bound $C_\kappa$ on a set of measure zero, offering greater flexibility in modeling applications.

\item[(G2)] \textbf{Integrability and non-degeneracy:} $\kappa\in L^1(I)$, i.e.,
\[
\|\kappa\|_{L^1(I)}:=\int_0^T|\kappa(\tau)|\,d\tau<\infty,
\]
and there exists a constant $\delta_\kappa>0$ such that
\[
\left|\int_0^T\kappa(\tau)\,d\tau\right|\ge\delta_\kappa.
\]
The latter inequality ensures that the total ``cumulative weight'' of the kernel is non-zero, excluding the trivial case where positive and negative contributions cancel out and the overall memory effect vanishes.
      
\item[(G3)] \textbf{Bounded variation:} $\kappa$ is a function of bounded variation on $I=[0,T]$. More precisely, the total variation of $\kappa$ on $I$ is defined as
\[
\operatorname{Var}_{I}(\kappa):=\sup_{\mathcal{P}}\sum_{i=1}^n|\kappa(\tau_i)-\kappa(\tau_{i-1})|,
\]
where the supremum is taken over all partitions $\mathcal{P} = \{\, \tau_0,\, \tau_1,\, \dots,\, \tau_n \,\}$ of $I$ satisfying 
\[
0=\tau_0<\tau_1<\cdots<\tau_n=T.
\]
We require $\operatorname{Var}_{I}(\kappa)<\infty$, i.e., $\kappa\in BV(I)$.
\end{description}

The collection of all generalized admissible kernels will be denoted by $\mathscr{K}_{\mathrm{gen}}$.
\end{defn}

\begin{rem}[Interpretation of the generalized conditions]\label{rmk:general_kernel_conditions}
We now elaborate on the mathematical content and modelling significance of each condition.
\medskip
\begin{itemize}[leftmargin=*, nosep]
\item \textbf{Essential boundedness (G1).} For a Lebesgue measurable function $f:E\to\mathbb{R}\cup\{\pm\infty\}$, the essential supremum is defined as
\[
\esssup_{x\in E} f(x)
:=
\inf\,
\left\{
M \in \mathbb{R} \cup \{+\infty\}
:
f(x) \le M
\ \text{for almost every } x \in E
\right\},
\]
or equivalently,
\[
\esssup_{x\in E} f(x)
=
\inf\,
\left\{
M \in \mathbb{R}
:
\mu\bigl( \{ x \in E : f(x) > M \} \bigr) = 0
\right\},
\]
with the convention that the infimum over an empty set is $+\infty$ (hence $\esssup f=+\infty$ when the set $\{M\in\mathbb{R}:\mu(\{x\in E:f(x)>M\})=0\}$ is empty). This notion relaxes the pointwise supremum by ignoring exceptional values on null sets, which is often more natural in applications.

\item \textbf{Non-degeneracy condition (G2).} The requirement 
\[
\left| \int_0^T \kappa(\tau)\,d\tau \right| \ge \delta_\kappa > 0
\]
is essential. If $\int_0^T\kappa(\tau)\,d\tau=0$, the memory term $\int_0^t\kappa(t-s)u(s)\,ds$ could reduce to a functional where positive and negative contributions cancel, potentially leading to a diminished or vanishing overall memory effect. The constant $\delta_\kappa$ excludes such degenerate cases, ensuring that the kernel retains a nontrivial cumulative influence.
    
\item \textbf{Bounded variation (G3).} Functions of bounded variation possess several key properties that are relevant in the context of memory kernels: They are differentiable almost everywhere, admit a Jordan decomposition as the difference of two monotone nondecreasing functions \cite{leonov1996total, chistyakov2010maps, ene1998decomposition}, and have at most countably many discontinuities, all of which are of the first kind (jump discontinuities). In the context of memory kernels, the bounded variation condition allows for jump-like changes, capturing abrupt adjustments or resets in the memory mechanism, while ensuring that the total oscillatory magnitude $\operatorname{Var}_{[0,T]}(\kappa)$ remains finite, thereby guaranteeing the stability and controllability of the weighted historical accumulation.
\end{itemize}
\end{rem}

\begin{rem}[Distinctive features and role of the generalized kernel class]\label{rmk:gen_kernel_features}
The generalized class $\mathscr{K}_{\mathrm{gen}}$ is designed to extend modeling capabilities in three principal directions:
\begin{itemize}[leftmargin=*, nosep]
\item \textbf{Sign flexibility:} Negative values of $\kappa$ are permitted, enabling the description of inhibitory effects or negative feedback mechanisms arising from past states.
\item \textbf{Relaxed regularity:} Kernels of bounded variation are admissible; global Lipschitz continuity is not imposed. Hence, kernels with (at most countably many) jump discontinuities are permitted, which is suitable for capturing abrupt changes, resets, or piecewise constant memory weights.
\item \textbf{Optional normalization:} The integral $\int_0^T\kappa(\tau)\,d\tau$ may take any non-zero value, with its magnitude controlled by the constant $\delta_\kappa$. This provides modeling freedom for systems where the overall memory intensity may vary or be adaptively regulated.
\end{itemize}
\medskip

The generalized class $\mathscr{K}_{\mathrm{gen}}$ introduced here serves as an interface for potential extensions: Under additional technical conditions, many of the results obtained for $\mathscr{K}_{\mathrm{reg}}$ are expected to carry over to $\mathscr{K}_{\mathrm{gen}}$, thereby covering an even broader spectrum of memory phenomena and application scenarios.
\end{rem}

\subsection{Stationary memory kernels and their representation}\label{subsec:stationary_kernel}

In many physical, biological, and engineering systems \cite{cortazar2021heat, cortazar2024asymptotic, wangattractor}, the manner in which memory weighting varies frequently depends upon the time elapsed since a past event, as opposed to the absolute moment at which that event occurred. This characteristic, referred to as time-shift stationarity, offers a natural and widely adopted simplification for modeling memory processes. Related studies in structural mechanics, adaptive algorithms, and time series analysis \cite{mori2018dissipative, sunthrayuth2025inertial, cheng2015inverse}, as well as in bio-inspired neural networks and adaptive signal decomposition \cite{nurutdinov2024bio, dong2026adaptive}, further illustrate the broad relevance of memory-related modeling. On the basis of the kernel classes introduced in the preceding subsections, we now provide a formal definition of stationary memory kernels.

\begin{defn}[Stationary memory kernel]\label{def:stationary_memory_kernel}
Let $\kappa : I \to \mathbb{R}$ be a mathematically admissible kernel in the sense of Definition~\ref{def:math_admissible_kernel} (more specifically, belonging to $\mathscr{K}_{\mathrm{reg}}$ or $\mathscr{K}_{\mathrm{gen}}$). The associated stationary memory kernel $K_\kappa:\Omega\to\mathbb{R}$ is defined by
\begin{equation}\label{eq:stationary_kernel_def}
K_\kappa(t,s):=\kappa(t-s),\qquad 0\le s\le t\le T.
\end{equation}
Consequently, the weight attributed to a historical state is determined entirely by the time lag
\[
\tau:=t-s\in[0,T],
\]
and does not depend on the particular historical instant $s$ or the current moment $t$.
\end{defn}

\begin{rem}[Interpretation of the stationarity assumption]\label{rmk:stationarity_assumption}
The relation 
\[
K_\kappa(t,s) = \kappa(t-s)
\]
encapsulates two fundamental aspects of stationary memory:
\begin{description}[leftmargin=*, nosep, font=\normalfont]
\item[(1)] \textbf{Time-translation invariance:} The memory characteristics of the system remain unchanged under shifts of the time axis; translating the timeline by any amount leaves the memory profile unaltered.
\item[(2)] \textbf{Dependence on relative time:} The weighting factor is determined exclusively by the elapsed time $\tau=t-s$ since the historical event, and is independent of the specific moment $s$ at which the event occurred. This reflects the principle that the influence of a past event is governed solely by how much time has passed, without reference to absolute calendar time.
\end{description}

The stationarity assumption is adopted throughout this work as a foundational premise for the theoretical development. It captures the characteristic feature that the influence of past events varies with the passage of time, while providing a clear and coherent structure for the analysis that follows.
\end{rem}

\begin{rem}[Notational convention]\label{rmk:stationary_notation}
Unless explicitly stated otherwise, the term ``kernel'' shall refer to a stationary memory kernel 
\[
K_\kappa(t,s)=\kappa(t-s)
\] 
generated by some $\kappa\in\mathscr{K}_{\mathrm{reg}}$ according to Definition~\ref{def:stationary_memory_kernel}. Depending on the context, we shall employ interchangeably the binary notation $K_\kappa(t,s)$ and the univariate notation $\kappa(\tau)$, as they are essentially equivalent under the stationarity hypothesis. All subsequent developments are conducted within this stationary framework, thereby ensuring clarity and consistency throughout the presentation. For regular admissible kernels $\kappa \in \mathscr{K}_{\mathrm{reg}}$, the range of $K_\kappa$ is restricted to $[0,\infty)$; for generalized admissible kernels $\kappa \in \mathscr{K}_{\mathrm{gen}}$, the range is allowed to be $\mathbb{R}$ in accordance with the possible sign changes of $\kappa$.
\end{rem}

\subsection{Typical examples of kernel functions}\label{subsec:kernel_examples}

To render the abstract kernel classes introduced above more concrete and to demonstrate that they encompass classical and physically important memory phenomena, in this subsection, we presents three prototypical examples that are widely encountered in applications. These examples not only validate the reasonableness of Definitions~\ref{def:regular_admissible_kernel} and~\ref{def:general_admissible_kernel}, but also furnish concrete objects for the theoretical analysis that follows.

\begin{ex}[Exponential kernel in the regular class]\label{ex:exponential_kernel}
A representative example belonging to the regular admissible kernel class $\mathscr{K}_{\mathrm{reg}}$ is provided by the exponential kernel. Let $\alpha>0$ be a rate parameter and define
\begin{equation}\label{eq:exponential_kernel}
\kappa_{\alpha}^{\mathrm{reg}}(\tau) := \frac{\alpha e^{-\alpha\tau}}{1 - e^{-\alpha T}}, \qquad \tau\in I.
\end{equation}
This kernel possesses the following properties:
\begin{itemize}[leftmargin=*, nosep]
\item \textbf{Basic mathematical properties:} It satisfies all the conditions (M1)--(M3) in Definition~\ref{def:math_admissible_kernel} for mathematically admissible kernels. Thus, $\kappa_{\alpha}^{\mathrm{reg}}\ge 0$ and is continuous, and its integral is given~by
          \[
          \int_0^T \kappa_{\alpha}^{\mathrm{reg}}(\tau)\,d\tau = \frac{\alpha}{1 - e^{-\alpha T}} \int_0^T e^{-\alpha\tau}\,d\tau = 1.
          \]

\item \textbf{Verification of regularity:} This kernel belongs to the regular class $\mathscr{K}_{\mathrm{reg}}$. First, for any $\tau\in I$,
          \[
          0 < \frac{\alpha e^{-\alpha T}}{1 - e^{-\alpha T}} \le \kappa_{\alpha}^{\mathrm{reg}}(\tau) \le \frac{\alpha}{1 - e^{-\alpha T}} < \infty,
          \]
which confirms the uniform boundedness condition (R1). The normalization condition (R2) follows directly from the integral evaluation above. For Lipschitz continuity (R3), we compute the derivative
          \[
          \frac{d}{d\tau}\kappa_{\alpha}^{\mathrm{reg}}(\tau) = -\frac{\alpha^2 e^{-\alpha\tau}}{1 - e^{-\alpha T}},
          \]
whose absolute value is bounded on $I$ by
\[
\frac{\alpha^2}{1 - e^{-\alpha T}};
\]
consequently, $\kappa_{\alpha}^{\mathrm{reg}}$ is Lipschitz continuous with constant
\[
L_\kappa = \frac{\alpha^2}{1 - e^{-\alpha T}}.
\]

\item \textbf{Physical interpretation:} This kernel describes the classical ``exponential'' weighting pattern. The memory weight varies exponentially with the elapsed time $\tau$, and parameter $\alpha$ controls the rate of this variation: Larger values of $\alpha$ correspond to a more rapid decrease, giving prominence to recent history, while smaller $\alpha$ imply a more gradual variation, allowing more distant events to retain some influence. This model is a standard choice for describing short-term memory behavior.
\end{itemize}
\end{ex}

\begin{ex}[Power-law kernel in the generalized class]\label{ex:power_law_kernel_generalized}
As a concrete instance of the generalized admissible kernel class $\mathscr{K}_{\mathrm{gen}}$ (see Definition~\ref{def:general_admissible_kernel}), let $\gamma\in(0,1)$ and $\varepsilon>0$, and define
\begin{equation}\label{eq:power_law_kernel_generalized}
	\kappa_{\gamma,\varepsilon}^{\mathrm{gen}}(\tau) \coloneqq \frac{1-\gamma}{T^{1-\gamma}} \bigl( T - \tau + \varepsilon \bigr)^{-\gamma}, \qquad \tau\in I.
\end{equation}
This kernel exhibits the following properties:
\begin{itemize}[leftmargin=*, nosep]
\item \textbf{Essential boundedness (G1):} Since $\varepsilon>0$, the denominator satisfies $(T-\tau+\varepsilon)\ge\varepsilon>0$ for all $\tau\in I$. Consequently,
\[
0 < \frac{1-\gamma}{T^{1-\gamma}} \bigl(T + \varepsilon\bigr)^{-\gamma} \le \kappa_{\gamma,\varepsilon}^{\mathrm{gen}}(\tau) \le \frac{1-\gamma}{T^{1-\gamma}} \varepsilon^{-\gamma} < \infty,
\]
which shows that $\kappa_{\gamma,\varepsilon}^{\mathrm{gen}}$ is uniformly bounded on $I$. More precisely,
\[
\esssup_{\tau\in I} \left| \kappa_{\gamma,\varepsilon}^{\mathrm{gen}}(\tau) \right|
\le \frac{1-\gamma}{T^{1-\gamma}} \varepsilon^{-\gamma}.
\]
Taking 
\[
C_\kappa := \frac{1-\gamma}{T^{1-\gamma}} \varepsilon^{-\gamma},
\] 
we have 
\[
\esssup_{\tau\in I} |\kappa_{\gamma,\varepsilon}^{\mathrm{gen}}(\tau)| \le C_\kappa < \infty,
\] 
confirming that $\kappa_{\gamma,\varepsilon}^{\mathrm{gen}}\in L^\infty(I)$ and thus satisfies condition (G1).
    
\item \textbf{Integrability and non-degeneracy (G2):} Computing the integral explicitly,
\begin{align*}
	\int_0^T \kappa_{\gamma,\varepsilon}^{\mathrm{gen}}(\tau)\,d\tau
	&= \frac{1-\gamma}{T^{1-\gamma}} \int_0^T (T-\tau+\varepsilon)^{-\gamma}\,d\tau \\
	&= \frac{1-\gamma}{T^{1-\gamma}}
	\biggl[ -\frac{(T-\tau+\varepsilon)^{1-\gamma}}{1-\gamma} \biggr]_{\tau=0}^{\tau=T} \\
	&= \frac{(T+\varepsilon)^{1-\gamma} - \varepsilon^{1-\gamma}}{T^{1-\gamma}}.
\end{align*}
Since $\kappa_{\gamma,\varepsilon}^{\mathrm{gen}}(\tau) > 0$ for all $\tau\in I$, we have $|\kappa_{\gamma,\varepsilon}^{\mathrm{gen}}(\tau)| = \kappa_{\gamma,\varepsilon}^{\mathrm{gen}}(\tau)$, and therefore
\[
\|\kappa_{\gamma,\varepsilon}^{\mathrm{gen}}\|_{L^1(I)}
\,
=
\int_0^T \kappa_{\gamma,\varepsilon}^{\mathrm{gen}}(\tau)\,d\tau
=
\frac{(T+\varepsilon)^{1-\gamma} - \varepsilon^{1-\gamma}}{T^{1-\gamma}}
\,< \infty,
\]
which verifies the integrability condition $\kappa_{\gamma,\varepsilon}^{\mathrm{gen}}\in L^1(I)$. Moreover,
\[
\left| \int_0^T \kappa_{\gamma,\varepsilon}^{\mathrm{gen}}(\tau)\,d\tau \right|
=
\frac{(T+\varepsilon)^{1-\gamma} - \varepsilon^{1-\gamma}}{T^{1-\gamma}}
> 0.
\]
Taking 
\[
\delta_\kappa \coloneqq \frac{(T+\varepsilon)^{1-\gamma} - \varepsilon^{1-\gamma}}{T^{1-\gamma}},
\]
we obtain 
\[
\left| \int_0^T \kappa_{\gamma,\varepsilon}^{\mathrm{gen}}(\tau)\,d\tau \right| = \delta_\kappa > 0,
\]
thereby satisfying the non-degeneracy condition in (G2). A closer examination of this integral reveals an interesting feature: For any fixed $\varepsilon>0$, one has
\[
\int_0^T \kappa_{\gamma,\varepsilon}^{\mathrm{gen}}(\tau)\,d\tau < 1.
\]
To verify this inequality, consider the auxiliary function $f: [0,\infty) \to \mathbb{R}$ defined by
\[
f(z) := (1+z)^{1-\gamma} - z^{1-\gamma}, \qquad z \ge 0.
\]
With the change of variable $z = \varepsilon/T \ge 0$, the integral becomes
\[
\int_0^T \kappa_{\gamma,\varepsilon}^{\mathrm{gen}}(\tau)\,d\tau = f \left(\frac{\varepsilon}{T}\right).
\]
  
\textit{Proof of the inequality.} The function $f$ is continuous on $[0,\infty)$ and differentiable on $(0,\infty)$. Its derivative is given by
\[
f'(z) = (1-\gamma)\left[(1+z)^{-\gamma} - z^{-\gamma}\right], \qquad z > 0.
\]
Since $\gamma \in (0,1)$, the map $t \mapsto t^{-\gamma}$ is strictly decreasing on $(0,\infty)$; consequently, for any $z > 0$,
\[
(1+z)^{-\gamma} < z^{-\gamma},
\]
which implies $f'(z) < 0$ for all $z > 0$. Hence, $f$ is strictly decreasing on $(0,\infty)$. Now, fix an arbitrary $\varepsilon > 0$ and set $z_0 = \varepsilon/T > 0$. Applying the Mean Value Theorem on the interval $[0, z_0]$, there exists $\xi \in (0, z_0)$ such that
\[
f(z_0) - f(0) = f'(\xi) z_0.
\]
Noting that $f(0) = 1$ and that $f'(\xi) < 0$, we obtain
\[
f(z_0) = 1 + f'(\xi) z_0 < 1.
\]
Substituting back yields
\[
\int_0^T \kappa_{\gamma,\varepsilon}^{\mathrm{gen}}(\tau)\,d\tau = f\left(\frac{\varepsilon}{T}\right) < 1.
\]
  
\item \textbf{Bounded variation (G3):} The function $\kappa_{\gamma,\varepsilon}^{\mathrm{gen}}$ is monotonically increasing on $I$. Indeed, for any $0\le\tau_1<\tau_2\le T$, we have $(T-\tau_1+\varepsilon) > (T-\tau_2+\varepsilon)$, and since the exponent $-\gamma$ is negative, it follows that 
\[
(T-\tau_1+\varepsilon)^{-\gamma} < (T-\tau_2+\varepsilon)^{-\gamma};
\]
consequently $\kappa_{\gamma,\varepsilon}^{\mathrm{gen}}(\tau_1) < \kappa_{\gamma,\varepsilon}^{\mathrm{gen}}(\tau_2)$. By monotonicity, for any partition $\mathcal{P}:0=\tau_0<\tau_1<\cdots<\tau_n=T$ of $I$, the sum of absolute differences reduces to the difference of the endpoint values:
\[
\sum_{i=1}^n \left| \kappa_{\gamma,\varepsilon}^{\mathrm{gen}}(\tau_i) - \kappa_{\gamma,\varepsilon}^{\mathrm{gen}}(\tau_{i-1}) \right|
= \sum_{i=1}^n \bigl[ \kappa_{\gamma,\varepsilon}^{\mathrm{gen}}(\tau_i) - \kappa_{\gamma,\varepsilon}^{\mathrm{gen}}(\tau_{i-1}) \bigr]
= \kappa_{\gamma,\varepsilon}^{\mathrm{gen}}(T) - \kappa_{\gamma,\varepsilon}^{\mathrm{gen}}(0).
\]
This value is independent of the particular partition $\mathcal{P}$, and therefore the total variation is given by
\[
\operatorname{Var}_{[0,T]}(\kappa_{\gamma,\varepsilon}^{\mathrm{gen}})
=
\sup_{\mathcal{P}} \left( \kappa_{\gamma,\varepsilon}^{\mathrm{gen}}(T) - \kappa_{\gamma,\varepsilon}^{\mathrm{gen}}(0) \right)
=
\kappa_{\gamma,\varepsilon}^{\mathrm{gen}}(T) - \kappa_{\gamma,\varepsilon}^{\mathrm{gen}}(0)
< \infty,
\]
which establishes (G3).
    
\item \textbf{Relation to the regular class:} It is instructive to examine how $\kappa_{\gamma,\varepsilon}^{\mathrm{gen}}$ relates to the regular class $\mathscr{K}_{\mathrm{reg}}$. The kernel satisfies two of the three conditions defining $\mathscr{K}_{\mathrm{reg}}$, namely uniform boundedness (R1) and Lipschitz continuity (R3). 
\begin{itemize}[nosep]
\item \textbf{Uniform boundedness (R1):} As already shown in the verification of (G1), for any $\tau\in I$,
\[
0 < \frac{1-\gamma}{T^{1-\gamma}} \bigl(T + \varepsilon\bigr)^{-\gamma} \le \kappa_{\gamma,\varepsilon}^{\mathrm{gen}}(\tau) \le \frac{1-\gamma}{T^{1-\gamma}} \varepsilon^{-\gamma} < \infty,
\]
which confirms (R1).
 
\item \textbf{Lipschitz continuity (R3):} Differentiating $\kappa_{\gamma,\varepsilon}^{\mathrm{gen}}$ yields
\[
\frac{d}{d\tau}\kappa_{\gamma,\varepsilon}^{\mathrm{gen}}(\tau)
= \frac{\gamma(1-\gamma)}{T^{1-\gamma}} \bigl( T - \tau + \varepsilon \bigr)^{-\gamma-1},
\]
and this derivative is uniformly bounded on $I$; indeed,
\[
\sup_{\tau\in I} \left| \frac{d}{d\tau}\kappa_{\gamma,\varepsilon}^{\mathrm{gen}}(\tau) \right|
\le
\frac{\gamma(1-\gamma)}{T^{1-\gamma}} \, \varepsilon^{-\gamma-1}
< \infty.
\]
By the Mean Value Theorem, this implies that $\kappa_{\gamma,\varepsilon}^{\mathrm{gen}}$ is Lipschitz continuous on $I$ with Lipschitz constant
\[
L_\kappa = \frac{\gamma(1-\gamma)}{T^{1-\gamma}} \varepsilon^{-\gamma-1},
\]
thereby satisfying condition (R3).
\end{itemize}

As observed in the verification of (G2) above, the integral of $\kappa_{\gamma,\varepsilon}^{\mathrm{gen}}$ is strictly less than $1$ for any $\varepsilon>0$; consequently, the normalization condition (R2) is not satisfied. Hence, $\kappa_{\gamma,\varepsilon}^{\mathrm{gen}}$ does not belong to $\mathscr{K}_{\mathrm{reg}}$. (When $\varepsilon = 0$, the integral equals $1$, but in that case $\kappa_{\gamma,0}^{\mathrm{gen}}$ becomes unbounded at $\tau = T$, violating (R1) and therefore also lies outside $\mathscr{K}_{\mathrm{reg}}$.) This observation illustrates the purpose of the generalized class $\mathscr{K}_{\mathrm{gen}}$: it accommodates physically reasonable kernels that possess the core analytical properties (boundedness and regularity) of the regular class, yet are not normalized (or have adjustable total weight) and therefore reside outside $\mathscr{K}_{\mathrm{reg}}$.

\item \textbf{Physical interpretation:} For sufficiently small regularization parameter $\varepsilon > 0$, this kernel describes a weighting pattern that evolves slowly with the time lag, following the power law
\[
(T - \tau + \varepsilon)^{-\gamma}.
\]

The parameter $\gamma \in (0,1)$ controls the rate of this evolution. Consider two historical instants $0 \le \tau_1 < \tau_2 \le T$; their weight ratio is
\[
\frac{\kappa_{\gamma,\varepsilon}^{\mathrm{gen}}(\tau_2)}
{\kappa_{\gamma,\varepsilon}^{\mathrm{gen}}(\tau_1)}
=
\left( \frac{T - \tau_1 + \varepsilon}{T - \tau_2 + \varepsilon} \right)^{\gamma}.
\]
Since $T - \tau_1 + \varepsilon > T - \tau_2 + \varepsilon$, this ratio exceeds $1$, indicating that the weight varies with the elapsed time in a manner that assigns greater influence to more distant events. For fixed $\tau_1, \tau_2$ and $\varepsilon$, the ratio as a function of $\gamma$, namely
\[
\gamma \mapsto \bigl( \frac{T - \tau_1 + \varepsilon}{T - \tau_2 + \varepsilon} \bigr)^{\gamma},
\]
is strictly increasing on $(0,1)$ (because the base is greater than $1$). Thus, smaller $\gamma$ yield a ratio closer to $1$, corresponding to a more gradual change and allowing more distant events to retain noticeable influence, while larger $\gamma$ produce a steeper variation, further amplifying the influence of more distant events. Parameter $\varepsilon > 0$ serves a dual purpose: Mathematically, it eliminates the singularity at $\tau = T$, ensuring that the kernel fully satisfies the technical requirements of the generalized class; physically, it may be interpreted as a small adjustment to the weighting of the very recent past (where $\tau$ is close to $T$), reflecting a finite resolution in the system's perception of the present moment.
\end{itemize}
\medskip

Thus, $\kappa_{\gamma,\varepsilon}^{\mathrm{gen}} \in \mathscr{K}_{\mathrm{gen}}$, providing a typical and physically meaningful example of the generalized kernel~class.
\end{ex}

\begin{ex}[Finite-memory kernel in the generalized class]\label{ex:finite_memory_kernel}
Another illustrative example, which belongs to the generalized class $\mathscr{K}_{\mathrm{gen}}$ (see Definition~\ref{def:general_admissible_kernel}) but not to the regular class $\mathscr{K}_{\mathrm{reg}}$ (see Definition~\ref{def:regular_admissible_kernel}) (except in the limiting case $\Delta=T$), is provided by the finite-memory kernel. Let $\Delta\in(0,T]$ be a parameter representing the length of the memory window, and define
\begin{equation}\label{eq:finite_memory_kernel}
\kappa_{\Delta}^{\mathrm{fm}}(\tau) := 
\begin{cases}
\displaystyle \frac{1}{\Delta}, & 0 \le \tau \le \Delta,\\[8pt]
0, & \Delta < \tau \le T.
\end{cases}
\end{equation}
This kernel exhibits the following features:
\begin{itemize}[leftmargin=*, nosep]
\item \textbf{Basic mathematical properties:} It satisfies the basic conditions (M1)--(M3) required for mathematically admissible kernels. The function is clearly non‑negative, measurable, and its integral evaluates to
\[
\int_0^T \kappa_{\Delta}^{\mathrm{fm}}(\tau)\,d\tau = \int_0^{\Delta} \frac{1}{\Delta}\,d\tau = 1.
\]
    
\item \textbf{Relation to the regular class:} The kernel $\kappa_{\Delta}^{\mathrm{fm}}$ fulfills two of the three conditions characterizing $\mathscr{K}_{\mathrm{reg}}$, but fails to satisfy the Lipschitz continuity requirement (R3) when $\Delta<T$.
\begin{itemize}[nosep]
\item \textbf{Uniform boundedness (R1):} For every $\tau\in I$, one has 
\[
0\le\kappa_{\Delta}^{\mathrm{fm}}(\tau)\le 1/\Delta,
\]
so (R1) holds.

\item \textbf{Normalization (R2):} As shown above, 
\[
\int_0^T\kappa_{\Delta}^{\mathrm{fm}}(\tau)\,d\tau=1,
\]
hence (R2) is satisfied.

\item \textbf{Lipschitz continuity (R3):} When $\Delta<T$, the kernel possesses a jump discontinuity at $\tau=\Delta$. Considering the difference quotient
\[
\frac{\left| \kappa_{\Delta}^{\mathrm{fm}}(\Delta+\epsilon) - \kappa_{\Delta}^{\mathrm{fm}}(\Delta-\epsilon) \right|}{2\epsilon}
=
\frac{\frac{1}{\Delta}}{2\epsilon}
=
\frac{1}{2\Delta\epsilon},
\]
which becomes unbounded as $\epsilon\to0^+$, one sees that $\kappa_{\Delta}^{\mathrm{fm}}$ cannot be Lipschitz continuous on $I$; consequently (R3) is not fulfilled (in the limiting case $\Delta=T$, the kernel reduces to the constant function $1/T$ on $[0,T]$, which is Lipschitz continuous and actually belongs to $\mathscr{K}_{\mathrm{reg}}$).
\end{itemize}
\item \textbf{Membership in the generalized class:} Despite the lack of Lipschitz continuity, $\kappa_{\Delta}^{\mathrm{fm}}$ meets all three conditions defining the generalized class $\mathscr{K}_{\mathrm{gen}}$.
\begin{itemize}[nosep]
\item \textbf{Essential boundedness (G1):} Clearly 
\[
\esssup_{\tau\in I}|\kappa_{\Delta}^{\mathrm{fm}}(\tau)| = 1/\Delta < \infty.
\]

\item \textbf{Integrability and non-degeneracy (G2):} One has 
\begin{align*}
	\int_0^T \left| \kappa_{\Delta}^{\mathrm{fm}}(\tau) \right|\,d\tau &= 1, \\
	\left| \int_0^T \kappa_{\Delta}^{\mathrm{fm}}(\tau)\,d\tau \right| &= 1,
\end{align*}
so (G2) is satisfied (with $\delta_\kappa=1$).

\item \textbf{Bounded variation (G3):} The function $\kappa_{\Delta}^{\mathrm{fm}}$ is piecewise constant with a single jump of height $1/\Delta$ at $\tau=\Delta$. For any partition $\mathcal{P}$ of $I$, the sum of absolute differences cannot exceed this jump height, and by taking partitions that include points arbitrarily close to $\Delta$ from both sides, the total variation is seen to be exactly $1/\Delta$. Hence, $\operatorname{Var}_{[0,T]}(\kappa_{\Delta}^{\mathrm{fm}})=1/\Delta<\infty$, which establishes~(G3).
\end{itemize}

Therefore $\kappa_{\Delta}^{\mathrm{fm}}\in\mathscr{K}_{\mathrm{gen}}$.
    
\item \textbf{Physical interpretation:} This kernel models a system with a strict finite memory horizon: Only the most recent $\Delta$ time units are retained, while everything older is completely forgotten. Such a weighting pattern arises naturally in digital systems, devices with limited storage capacity, or sliding-window filters used in signal processing and real-time applications. Parameter $\Delta$ governs the memory capacity: As $\Delta\to0^+$, the memory window shrinks to zero, so that in the limit, the system retains information only about the present instant, effectively approaching a memoryless regime; when $\Delta=T$, the kernel is constant on the whole interval, representing uniform weighting of the entire past; for intermediate values $\Delta\in(0,T)$, the system possesses a finite but non-trivial memory.
\end{itemize}
\end{ex}

\begin{rem}[Significance and selection of the examples]\label{rmk:examples_significance}
The three examples presented above illustrate fundamental memory patterns that are both physically relevant and mathematically tractable. The exponential kernel $\kappa_{\alpha}^{\mathrm{exp}}$ (Example~\ref{ex:exponential_kernel}) belongs to the regular class $\mathscr{K}_{\mathrm{reg}}$ and captures short-memory behavior, where the weighting assigned to past events diminishes with increasing time lag. The power-law kernel $\kappa_{\gamma,\varepsilon}^{\mathrm{gen}}$ (Example~\ref{ex:power_law_kernel_generalized}) represents a contrasting long-memory pattern: due to its increasing nature as a function of the time lag, it assigns greater weight to more distant past events, thereby complementing the short-memory scenario. The finite-memory kernel $\kappa_{\Delta}^{\mathrm{fm}}$ (Example~\ref{ex:finite_memory_kernel}) belongs to the generalized class $\mathscr{K}_{\mathrm{gen}}$ and illustrates how kernels with discontinuities can still be accommodated within the framework while retaining essential analytical properties such as bounded variation. Together, these examples demonstrate the flexibility of the proposed kernel classes in capturing a range of memory phenomena.
\end{rem}

\subsection{Essential integral estimates for admissible kernels}\label{subsec:kernel_integral_estimates}

Integral estimates for kernel functions play a fundamental role in the analysis of memory-dependent operators. In this subsection, we establish a set of basic integral control inequalities. The estimates are formulated primarily for regular admissible kernels $\kappa\in\mathscr{K}_{\mathrm{reg}}$ (see Definition~\ref{def:regular_admissible_kernel}); their favorable properties, in particular boundedness and normalization, lead to particularly concise forms.

\begin{lem}[Integral control estimates for regular kernels]\label{lem:kernel_integral_estimates}
Let $\kappa\in\mathscr{K}_{\mathrm{reg}}$ be a regular admissible kernel in the sense of Definition~\ref{def:regular_admissible_kernel}. Throughout, $I$ denotes the fixed time interval $[0,T]$. Then the following estimates hold:
\begin{description}[leftmargin=*, nosep, font=\normalfont]
\item[(1)] \textbf{Cumulative weight estimate:} For every $t\in I$,
\begin{equation}\label{eq:cumulative_weight}
\int_0^t \kappa(t-s)\,ds \le 1.
\end{equation}
In particular, the normalization condition $\int_0^T\kappa(\tau)\,d\tau = 1$ is satisfied.

\item[(2)] \textbf{Weighted supremum estimate:} For every $f\in\mathcal{C}(I)$ and every $t\in I$,
\begin{equation}\label{eq:weighted_supremum}
\int_0^t \kappa(t-s) |f(s)|\,ds\le \|f\|_{\infty}.
\end{equation}
Here, $\|f\|_{\infty}:=\sup_{s\in I}|f(s)|$ denotes the usual supremum norm on $\mathcal{C}(I)$.

\item[(3)] \textbf{Integral control of differences:} For every $f,g\in\mathcal{C}(I)$ and every $t\in I$,
\begin{equation}\label{eq:difference_control}
	\left| \int_0^t \kappa(t-s)\,\bigl( f(s)-g(s) \bigr)\,ds \right| \le \|f-g\|_{\infty}.
\end{equation}
\end{description}
\end{lem}

\begin{proof}
(1) By the change of variables $\tau = t-s$, one obtains
\[
\int_0^t \kappa(t-s)\,ds = \int_0^t \kappa(\tau)\,d\tau.
\]
The non-negativity of $\kappa$ (condition (R1) in Definition~\ref{def:regular_admissible_kernel}) together with the normalization condition~(R2) yields
\[
\int_0^t \kappa(\tau)\,d\tau \le \int_0^T \kappa(\tau)\,d\tau = 1.
\]
(2) From $\kappa\ge0$ and the definition of the supremum norm,
\[
\int_0^t \kappa(t-s) \, |f(s)| \, ds
\le
\sup_{s\in[0,t]} |f(s)| \int_0^t \kappa(t-s) \, ds
\le
\|f\|_{\infty} \int_0^t \kappa(\tau) \, d\tau.
\]
Applying estimate (1) gives $\|f\|_{\infty}\cdot 1 = \|f\|_{\infty}$.

\noindent(3) Applying estimate (2) with $f$ replaced by $f-g$ yields
\[
\left| \int_0^t \kappa(t-s)\,\bigl( f(s)-g(s) \bigr)\,ds \right|
\le
\int_0^t \kappa(t-s)\,|f(s)-g(s)|\,ds
\le
\|f-g\|_{\infty}.
\]
The first inequality follows from the triangle inequality for integrals together with the non-negativity of $\kappa$.
\end{proof}

\begin{prop}[Integral estimates for generalized kernels]\label{prop:generalized_estimates}
Let $\kappa\in\mathscr{K}_{\mathrm{gen}}$ be a generalized admissible kernel as introduced in Definition~\ref{def:general_admissible_kernel}. Then the following estimates hold:
\begin{description}[leftmargin=*, nosep, font=\normalfont]
\item[(1)] \textbf{Non-negative case:} If $\kappa(\tau)\ge 0$ for almost every $\tau\in I$, then for every $f\in\mathcal{C}(I)$ and every $t\in I$,
\begin{equation}\label{eq:gen_nonneg_estimate}
	\int_0^t \kappa(t-s)\,|f(s)|\,ds
	\le \|f\|_{\infty} \int_0^T \kappa(\tau)\,d\tau.
\end{equation}
 
\item[(2)] \textbf{General case (possibly sign-changing):} For every $f,g\in\mathcal{C}(I)$ and every $t\in I$,
\begin{align}
	\left|\int_0^t \kappa(t-s)\,f(s)\,ds\right| &\le \|f\|_{\infty} \int_0^T |\kappa(\tau)|\,d\tau, \label{eq:gen_abs_estimate}\\
	\left|\int_0^t \kappa(t-s)\,\bigl(f(s)-g(s)\bigr)\,ds\right| &\le \|f-g\|_{\infty} \int_0^T |\kappa(\tau)|\,d\tau. \label{eq:gen_diff_estimate}
\end{align}
\end{description}
\end{prop}

\begin{proof}
We prove each case separately.

\noindent \textit{Proof of (1).} Assume $\kappa(\tau)\ge0$ for almost every $\tau\in I$. Then $|\kappa(\tau)|=\kappa(\tau)$ almost everywhere. Using the non-negativity of $\kappa$ and the definition of the supremum norm,
\[
\int_0^t \kappa(t-s)\,|f(s)|\,ds
\le
\sup_{s\in[0,t]} |f(s)| \int_0^t \kappa(t-s)\,ds
\le
\|f\|_{\infty} \int_0^t \kappa(\tau)\,d\tau.
\]
Since $\kappa\ge0$, one has 
\[
\int_0^t \kappa(\tau)\,d\tau \le \int_0^T \kappa(\tau)\,d\tau,
\] 
which yields the desired estimate.

\noindent \textit{Proof of (2).} For arbitrary $\kappa  \in  \mathscr{K}_{\mathrm{gen}}$, condition (G2) in Definition~\ref{def:general_admissible_kernel} guarantees $\kappa  \in  L^1(I)$ with 
\[
\int_0^T|\kappa(\tau)|d\tau<\infty.
\] 
For any $t\in I$, the triangle inequality for integrals gives
\[
\left| \int_0^t \kappa(t-s)\,f(s)\,ds \right|
\le \int_0^t |\kappa(t-s)|\,|f(s)|\,ds
\le \|f\|_{\infty} \int_0^t |\kappa(t-s)|\,ds.
\]
By the change of variables $\tau=t-s$, the last integral equals 
\[
\int_0^t |\kappa(\tau)|\,d\tau \le \int_0^T |\kappa(\tau)|\,d\tau,
\] 
establishing \eqref{eq:gen_abs_estimate}. Estimate \eqref{eq:gen_diff_estimate} follows by applying the same argument to $f-g$ in place of $f$.
\end{proof}

\begin{rem}
The estimates established in this subsection provide basic tools for handling integral expressions involving kernel functions. Lemma~\ref{lem:kernel_integral_estimates} presents three fundamental estimates for regular kernels $\kappa\in\mathscr{K}_{\mathrm{reg}}$, relying on their non-negativity, boundedness, and normalization. Proposition~\ref{prop:generalized_estimates} extends these estimates to the broader class $\mathscr{K}_{\mathrm{gen}}$, where the $L^1$ norm $\int_0^T|\kappa(\tau)|d\tau$ naturally appears in place of the constant $1$. In the non‑negative case, the factor reduces to $\int_0^T\kappa(\tau)d\tau$, consistent with the structure of the regular estimates. Together, these results illustrate how the integral estimates adapt to the regularity properties of the kernel, offering flexibility for the regular and generalized settings considered in this work.
\end{rem}

\section{Adaptive sensitivity functions: definition, construction, and basic properties}\label{sec:adaptive_sensitivity_framework}

We now turn to the second layer of the framework: the adaptive sensitivity functions, which modulate the kernel weights according to the state of the trajectory. This layer introduces a state-dependent mechanism that works in tandem with the kernel hierarchy established in the previous section. The kernel-theoretic framework developed in Section~\ref{sec:kernel_theory} provides a hierarchical classification of memory kernels based on their regularity and integrability properties, capturing the temporal weighting patterns through stationary kernels of the form $K_\kappa(t,s)=\kappa(t-s)$. This framework, with its emphasis on the temporal dimension of memory, characterizes the weight assigned to a past state $f(s)$ as a function $\kappa(t-s)$ of the elapsed time. The adaptive sensitivity function introduced below retains this temporal weighting as a foundational layer, while incorporating an additional mechanism that modulates the weight according to the state value $f(s)$. Specifically, we augment the kernel $\kappa$ with a sensitivity factor $\Lambda(s,f(s))$, yielding a composite memory weight of the form $\Lambda(s,f(s))\kappa(t-s)$. This construction gives mathematical expression to two key capabilities: The capacity for selective retention of historically salient events, and the capacity for differentiated assessment based on state-dependent importance.

\subsection{Mathematical framework of adaptive sensitivity functions}\label{subsec:adaptive_sensitivity_framework}

The memory formalism developed in the preceding section establishes a foundational temporal weighting mechanism, wherein the influence of a past state $f(s)$ is mediated through the kernel $\kappa$ via $K_\kappa(t,s)=\kappa(t-s)$. This construction captures the key idea that the weight attributed to a historical instant depends naturally on the time elapsed. The adaptive framework proposed below retains this temporal weighting as a core component, while extending it to incorporate sensitivity to the state values. Specifically, we introduce a modulation factor $\Lambda(s,f(s))$ that works in tandem with $\kappa(t-s)$, yielding a composite weight $\Lambda(s,f(s))\kappa(t-s)$. The following definition formalizes this extended object.

\begin{defn}[Adaptive sensitivity function]\label{def:adaptive_sensitivity}
Let $I=[0,T]$ denote the fixed time horizon. A function $\Lambda: I\times\mathbb{R}\to[0,\infty)$ is termed an adaptive sensitivity function if it satisfies the following four conditions:
\begin{description}[leftmargin=*, nosep, font=\normalfont]
\item[(AS1)] \textbf{Uniform boundedness:} There exist constants $0<\lambda_{\mathrm{min}}\le\lambda_{\mathrm{max}}<\infty$, such that
\[
\lambda_{\mathrm{min}}\le\Lambda(s,x)\le\lambda_{\mathrm{max}}\qquad\text{for all }(s,x)\in I\times\mathbb{R}.
\]
This condition precludes unbounded amplification and unphysical vanishing to zero, thereby maintaining the regularity essential for subsequent analytical developments.

\item[(AS2)] \textbf{Lipschitz continuity with respect to the state variable:} There exists a constant $L_{\Lambda}>0$, such~that
\[
|\Lambda(s,x)-\Lambda(s,y)|\le L_{\Lambda}|x-y|\qquad\text{for all }s\in I,\;x,y\in\mathbb{R}.
\]
The Lipschitz requirement guarantees that sensitivity adjustments occur gradually as the state varies, precluding abrupt jumps that would otherwise complicate physical interpretation and mathematical~treatment.

\item[(AS3)] \textbf{Measurability with respect to the temporal variable:} For every fixed $x\in\mathbb{R}$, the mapping $s\mapsto\Lambda(s,x)$ is Lebesgue measurable on $I$.
This measurability requirement is a fundamental prerequisite for the Lebesgue integration theory, ensuring that expressions, such as
\[
\int_0^t\Lambda(s,f(s))\kappa(t-s)|f(s)|\,ds
\]
are mathematically well-defined.

\item[(AS4)] \textbf{Positivity at the zero state:} There exists a constant $\varrho>0$, such that
\[
\Lambda(s,0)\ge\varrho\qquad\text{for all }s\in I.
\]
This condition excludes the degenerate case $\Lambda(s,0)=0$, which could compromise the distinctive status of the zero function in the functional-analytic framework.
\end{description}

The collection of all adaptive sensitivity functions satisfying conditions (AS1)--(AS4) will be denoted by $\mathscr{A}(I)$.

In this definition, the variable $x\in\mathbb{R}$ denotes a generic state value; when such a function is applied to a specific system trajectory $f\in\mathcal{C}(I)$, we write $\Lambda(s, f(s))$, where $f(s)\in\mathbb{R}$ is the value of $f$ at time $s$. This is precisely the form appearing in expressions such as
\[
\int_0^t \Lambda(s, f(s)) \kappa(t-s) |f(s)| \, ds,
\]
which will be central to the construction of the adaptive memory-dependent functional in the following subsection (see Definition~\ref{def:adaptive_memory_functional}).
\end{defn}

\begin{rem}[Systematic rationale underlying the sensitivity conditions]\label{rmk:sensitivity_conditions_design}
The four conditions encapsulated in Definition~\ref{def:adaptive_sensitivity} constitute a carefully balanced axiomatic system that reconciles mathematical tractability with modeling fidelity:
\begin{itemize}[leftmargin=*, nosep]
\item \textbf{Complementary roles of (AS1) and (AS4):} While condition (AS1) provides global bounds controlling the extreme excursions of $\Lambda$, condition (AS4) safeguards against vanishing sensitivity at the distinguished state $x \!=\! 0$. Together, they prevent numerical singularities and physical degeneracies, ensuring that the sensitivity mechanism remains operative across admissible~states.

\item \textbf{Analytical significance of (AS2):} Beyond its natural physical interpretation, sensitivity ought to vary continuously with the underlying state, the Lipschitz condition furnishes the regularity necessary for subsequent estimates involving compositions of $\Lambda$ with continuous functions. This property will be essential when examining the continuity properties of functionals defined through adaptive memory weights.

\item \textbf{Relation to the kernel hierarchy:} The adaptive sensitivity function $\Lambda$ operates in conjunction with the kernel classes $\mathscr{K}_{\mathrm{reg}}$ and $\mathscr{K}_{\mathrm{gen}}$ introduced in Section~\ref{sec:kernel_theory}. While those classes characterize the temporal weighting patterns through $\kappa$, the present construction introduces a complementary mechanism that modulates intensity according to state values. The composite expression $\Lambda(s,f(s))\kappa(t-s)$ thereby encodes the objective, time-driven weighting of past events and the system's capacity to assess their significance based on content. This dual structure enables the description of nuanced memory phenomena such as heightened retention of salient episodes or suppression of routine fluctuations.

\item \textbf{Balance between generality and analytical utility:} Each condition has been formulated to be as weak as possible while still providing the essential properties required for the subsequent development of the theory. Stronger requirements (e.g., differentiability with respect to $x$, joint continuity, or monotonicity) have been deliberately avoided to preserve maximal generality, thereby enabling $\mathscr{A}(I)$ to accommodate a broad spectrum of potential applications.
\end{itemize}

The axiomatic framework established here thus achieves a synthesis of mathematical rigor and physical plausibility, laying a foundation for the construction of adaptive memory structures.
\end{rem}

\begin{rem}[Modular interplay between kernel classes and sensitivity functions]\label{rmk:sensitivity_and_kernel_classes}
A distinctive feature of the present framework lies in the modular separation between the temporal weighting mechanism, encoded by $\kappa\in\mathscr{K}_{\mathrm{reg}}\cup\mathscr{K}_{\mathrm{gen}}$, and the state-dependent modulation, encoded by $\Lambda\in\mathscr{A}(I)$. This conceptual division offers considerable flexibility: One may select any kernel from the hierarchy established in Section~\ref{sec:kernel_theory} and combine it with any sensitivity function satisfying Definition~\ref{def:adaptive_sensitivity}, tailoring the resulting composite memory structure to specific modeling requirements while preserving the benefits of a unified theoretical treatment.
\end{rem}

\subsection{A constructive example: sensitivity with historical deviation accumulation}
\label{subsec:historical_deviation_example}

The axiomatic framework presented in Definition~\ref{def:adaptive_sensitivity} delineates a general class of functions $\mathscr{A}(I)$ characterized by boundedness, Lipschitz regularity in the state variable, measurability in time, and positivity at the zero state. Within this abstract setting, a question arises: Does $\mathscr{A}(I)$ contain functions whose sensitivity mechanism reflects more intricate dependence on the historical evolution of the state, beyond instantaneous values? In this subsection, we address this question by exhibiting a concrete construction in which the sensitivity at a given state incorporates a weighted accumulation of past deviations from a reference trajectory. This construction serves to illustrate the scope of the axiomatic framework and to provide a tangible object for subsequent analysis.

\begin{ex}[Adaptive sensitivity based on historical deviations]
\label{ex:historical_deviation_adaptive}
Let $I=[0,T]$ and let $r\in\mathcal{C}(I)$ be a given continuous function, interpreted as a reference trajectory representing the nominal or expected evolution of the system. Choose parameters $\alpha_0>0$ governing the persistence of historical influence, $\gamma_0>0$ controlling the sensitivity to instantaneous deviations, and $\beta_0\ge 0$ modulating the strength of the historical feedback. 

For any given trajectory $f\in\mathcal{C}(I)$, define the function $\Lambda_f: I \to [0,\infty)$ by
\begin{equation}\label{eq:lambda_f}
	\Lambda_f(s) := \lambda_{\min} + (\lambda_{\max} - \lambda_{\min}) \,
	\frac{
		\tanh\bigl(\gamma_0 |f(s) - r(s)|\bigr)
	}{
		1 + \beta_0 \displaystyle\int_0^s e^{-\alpha_0(s-\tau)} \,
		\tanh\bigl(\gamma_0 |f(\tau) - r(\tau)|\bigr) \, d\tau
	},
\end{equation}
where the constants satisfy $0<\lambda_{\min}<\lambda_{\max}<\infty$, and $\tanh(z)=(e^z-e^{-z})/(e^z+e^{-z})$ denotes the hyperbolic tangent function.

The object \(\Lambda_f\) defined above has two distinct interpretations depending on the parameter \(\beta_0\): For \(\beta_0 = 0\), it reduces to a function \(\Lambda(s,x) \in \mathscr{A}(I)\); for \(\beta_0 > 0\), it is trajectory-dependent and can be regarded as an operator \(f \mapsto \Lambda_f\).\\
\textbf{Interpretation.} For a fixed trajectory $f$, the denominator integral in \eqref{eq:lambda_f} accumulates the actual weighted deviations of $f$ from the reference $r$ over the interval $[0,s]$. This cumulative historical deviation then modulates the sensitivity at time $s$ through the denominator. The construction thus encodes a genuine historical feedback mechanism: If $f$ has persistently deviated from $r$ in the past, the sensitivity $\Lambda_f(s)$ is attenuated, even if the current deviation $|f(s)-r(s)|$ is large.

When $\beta_0 = 0$, the construction reduces to a purely instantaneous sensitivity that can be represented as a function $\Lambda\in\mathscr{A}(I)$ given by
\[
\Lambda(s,x) = \lambda_{\min} + (\lambda_{\max} - \lambda_{\min}) \tanh(\gamma_0 |x - r(s)|),
\]
so that for any trajectory $f$, we have $\Lambda_f(s) = \Lambda(s, f(s))$. Here $x\in\mathbb{R}$ denotes a generic state value, while $f(s)$ is the specific value of $f$ at time $s$. A detailed verification that this $\Lambda$ indeed belongs to $\mathscr{A}(I)$ is provided in Corollary~\ref{cor:beta_zero_case} below. For $\beta_0 > 0$, the sensitivity acquires genuine historical dependence and is most naturally viewed as an operator mapping each trajectory $f$ to the function $\Lambda_f$ defined above.

The construction incorporates several interrelated features:
\begin{itemize}[leftmargin=*, nosep]
\item \textbf{Bounded quantification of instantaneous deviations.}
The numerator $\tanh(\gamma_0|f(s)-r(s)|)$ maps the absolute deviation $|f(s)-r(s)|\in[0,\infty)$ to the bounded interval $[0,1)$. The saturation property of the hyperbolic tangent tempers the influence of extreme deviations while preserving ordinal information; parameter $\gamma_0$ controls the slope of this function near the origin, i.e., the sensitivity to small deviations; it thus governs the gain of the instantaneous response.

\item \textbf{Exponentially weighted accumulation of historical deviations.} Define the weighted historical deviation at time $s$ by
\[
\mathcal{D}_f(s) := \int_0^s e^{-\alpha_0(s-\tau)} \,
\tanh\bigl( \gamma_0 |f(\tau) - r(\tau)| \bigr) \, d\tau.
\]
This quantity aggregates past deviations with an exponential weight $e^{-\alpha_0(s-\tau)}$ that diminishes as the historical time $\tau$ recedes from the present $s$. This weighting pattern reflects the principle that more recent episodes exert stronger influence on the current sensitivity;
parameter $\alpha_0$ determines the temporal extent of this memory: Larger $\alpha_0$ corresponds to shorter persistence (emphasizing recent deviations), while smaller $\alpha_0$ permits more persistent historical influence.
    
\item \textbf{Adaptive modulation through a feedback mechanism.} The denominator introduces a self-regulating effect: If the trajectory $f$ has historically exhibited sustained deviations from the reference trajectory over an extended period, the accumulated integral becomes large, thereby attenuating the growth of $\Lambda_f(s)$. This negative feedback may be viewed as a mathematical analog of adaptation phenomena observed in certain biological and engineered systems, where sustained anomalous stimuli lead to diminished responsiveness, thereby preventing excessive sensitivity to recurrent deviations \cite{allana2025towards, ruiz2022spectral, tsuda2026some}.

\hspace*{1em} A particularly illustrative scenario arises when the trajectory $f$ closely follows the reference $r$. In this case, $|f(s)-r(s)|$ and $|f(\tau)-r(\tau)|$ are small for all relevant times. For small arguments, the hyperbolic tangent satisfies $\tanh(z) = z + o(z)$ as $z\to 0$, so that $\tanh(\gamma_0|f(s)-r(s)|)$ is well approximated by $\gamma_0|f(s)-r(s)|$ up to higher-order corrections; an analogous statement holds for the integrand in the denominator. Consequently, the numerator and the integral term in the denominator are small, and $\Lambda_f(s)$ remains close to its minimum value $\lambda_{\min}$, reflecting a low sensitivity to near-reference behavior. This limiting behavior underscores the consistency of the construction: trajectories that adhere to the nominal evolution receive little modulation, while significant deviations are required to elevate the sensitivity.

\hspace*{1em}Parameter $\beta_0$ controls the strength of this feedback: When $\beta_0=0$, the construction reduces to the purely instantaneous case described above, while positive $\beta_0$ introduces historical dependence.
    
\item \textbf{Analytical tractability.} 
The hyperbolic tangent function is Lipschitz continuous on $[0,\infty)$ with Lipschitz constant $1$, strictly increasing, and satisfies $0\le\tanh(z)<1$ for all $z\ge0$. The composition $\tanh(\gamma_0|\cdot|)$ inherits these properties. The exponential kernel $e^{-\alpha_0(s-\tau)}$ guarantees that the integral term is well-defined and preserves the regularity of the integrand. These properties facilitate the verification that, for each fixed $f$, the function $\Lambda_f$ exhibits characteristics analogous to conditions (AS1)--(AS4) in Definition~\ref{def:adaptive_sensitivity}.
\end{itemize}

This construction extends the notion of memory sensitivity from a static or instantaneously determined mapping to one that incorporates information about the historical behavior of the system trajectory. The resulting sensitivity function captures a dynamic interplay between present deviation and past experience, suggesting some possible modeling directions for systems where responsiveness adapts based on accumulated history.
\end{ex}

\begin{thm}[Basic properties of $\boldsymbol{\Lambda_f}$]
\label{thm:lambda_f_properties}
Under the assumptions of Example~\ref{ex:historical_deviation_adaptive}---namely, $r\in\mathcal{C}(I)$, $\alpha_0>0$, $\gamma_0>0$, $\beta_0\ge0$, and $0<\lambda_{\min}<\lambda_{\max}<\infty$---for any fixed trajectory $f\in\mathcal{C}(I)$, the function $\Lambda_f: I\to[0,\infty)$ defined by \eqref{eq:lambda_f} satisfies the following properties:
\begin{description}[leftmargin=*, nosep, font=\normalfont] 
    \item[(P1)] \textbf{Uniform boundedness:} $\lambda_{\min} \le \Lambda_f(s) \le \lambda_{\max}$ for all $s\in I$.
    
    \item[(P2)] \textbf{Lipschitz-type estimate with respect to the supremum norm:} 
For any two trajectories $f,g\in\mathcal{C}(I)$,
\begin{equation}\label{eq:lambda_f_lipschitz_sup}
\|\Lambda_f - \Lambda_g\|_\infty \le L_\Lambda \|f-g\|_\infty,
\end{equation}
where $\|h\|_\infty := \sup_{t\in I}|h(t)|$ and
\[
L_\Lambda = (\lambda_{\max}-\lambda_{\min})\gamma_0\left(1 + \frac{\beta_0(1-e^{-\alpha_0 T})}{\alpha_0}\right).
\]
    
\item[(P3)] \textbf{Continuity (hence measurability) in time:} For each fixed trajectory $f\in\mathcal{C}(I)$, the function $s\mapsto\Lambda_f(s)$ is continuous on $I$, and consequently Lebesgue measurable.
    
\item[(P4)] \textbf{Positivity at zero:} 
For the zero trajectory $f\equiv0$, one has $\Lambda_0(s) \ge \lambda_{\min} > 0$ for all $s\in I$.
\end{description}

These properties are the precise analogues of conditions (AS1)--(AS4) in Definition~\ref{def:adaptive_sensitivity}, now formulated for the operator-induced function $\Lambda_f$.
\end{thm}

\begin{proof}
Fix an arbitrary $f\in\mathcal{C}(I)$.

\noindent \textit{Proof of (P1): uniform boundedness.}
For any $s\in I$, the properties of the hyperbolic tangent yield
\[
0 \le \tanh\bigl(\gamma_0|f(s)-r(s)|\bigr) < 1,\qquad
0 \le \tanh\bigl(\gamma_0|f(\tau)-r(\tau)|\bigr) < 1 \quad \text{for all } \tau \in [0,s].
\]
Since $\beta_0\ge0$ and $e^{-\alpha_0(s-\tau)}\ge0$, the denominator satisfies
\begin{equation}\label{eq:denom_lower_bound}
	1 \le 1 + \beta_0 \int_0^s e^{-\alpha_0(s-\tau)} \,
	\tanh\bigl(\gamma_0|f(\tau)-r(\tau)|\bigr) \, d\tau.
\end{equation}
Consequently,
\begin{equation}\label{eq:fraction_bound}
	0 \le
	\frac{
		\tanh\bigl( \gamma_0 |f(s) - r(s)| \bigr)
	}{
		1 + \beta_0 \displaystyle\int_0^s e^{-\alpha_0(s-\tau)} \,
		\tanh\bigl( \gamma_0 |f(\tau) - r(\tau)| \bigr) \, d\tau
	}
	< 1.
\end{equation}
Substituting into the definition of $\Lambda_f$ and using $0<\lambda_{\min}<\lambda_{\max}$ gives
\[
\lambda_{\min}\le\Lambda_f(s)\le\lambda_{\min}+(\lambda_{\max}-\lambda_{\min})\cdot1=\lambda_{\max},
\]
which establishes (P1).

\noindent \textit{Proof of (P2): Lipschitz-type estimate with respect to the supremum norm.}
We now establish the Lipschitz-type estimate for $\Lambda_f$. For any two trajectories $f,g\in\mathcal{C}(I)$ and any $s\in I$, write
\[
\Lambda_f(s)=\lambda_{\min}+(\lambda_{\max}-\lambda_{\min})G_f(s),\qquad
\Lambda_g(s)=\lambda_{\min}+(\lambda_{\max}-\lambda_{\min})G_g(s),
\]
where
\[
G_f(s) \coloneqq
\frac{
	\tanh\bigl( \gamma_0 \, |f(s) - r(s)| \bigr)
}{
	1 + \beta_0 \displaystyle\int_0^s e^{-\alpha_0(s-\tau)} \,
	\tanh\bigl( \gamma_0 \, |f(\tau) - r(\tau)| \bigr) \, d\tau
},
\]
and $G_g(s)$ is defined analogously with $g$ in place of $f$. To establish this estimate, we aim to control $|G_f(s)-G_g(s)|$ by $\|f-g\|_\infty$.

To this end, we first establish a fundamental Lipschitz estimate for the hyperbolic tangent function.

\begin{lem}[Lipschitz continuity of $\tanh(\gamma_0|\cdot-p|)$]\label{lem:tanh_abs_lipschitz}
	
Let $p\in\mathbb{R}$ be fixed and let $\gamma_0>0$. Then the function $\psi_{p,\gamma_0}(x):=\tanh(\gamma_0|x-p|)$ is $\gamma_0$-Lipschitz continuous with respect to $x$; that is, for all $x,y\in\mathbb{R}$,
\begin{equation}\label{eq:tanh_abs_lipschitz}
	\bigl| \tanh(\gamma_0|x-p|) - \tanh(\gamma_0|y-p|) \bigr|
	\le \gamma_0\,|x-y|.
\end{equation}
\end{lem}

\begin{proof}
We first compute the derivative of the hyperbolic tangent:
\[
\frac{d}{dz}\tanh(z)=\operatorname{sech}^2(z)=\frac{1}{\cosh^2(z)}.
\]
For any real $z$, by the properties of the exponential function and the arithmetic-geometric mean inequality,
\[
\cosh(z)=\frac{e^z+e^{-z}}{2}\ge\sqrt{e^z\cdot e^{-z}}=1,
\]
with equality if and only if $e^z=e^{-z}$, i.e., $z=0$. Hence, $\cosh^2(z)\ge1$, and consequently
\begin{equation}\label{eq:sech_bound}
0<\operatorname{sech}^2(z)=\frac{1}{\cosh^2(z)}\le1\qquad(\forall z\in\mathbb{R}).
\end{equation}

Now consider $\psi_{p,\gamma_0}(x)=\tanh(\gamma_0|x-p|)$. Write it in piecewise form:
\[
\psi_{p,\gamma_0}(x)=
\begin{cases}
	\tanh\bigl(\gamma_0(p-x)\bigr), & x \le p,\\[6pt]
	\tanh\bigl(\gamma_0(x-p)\bigr), & x \ge p.
\end{cases}
\]
Set
\[
\psi_-(x)=\tanh\bigl(\gamma_0(p-x)\bigr)\;(x\le p),\qquad
\psi_+(x)=\tanh\bigl(\gamma_0(x-p)\bigr)\;(x\ge p),
\]
so that $\psi_-(p)=\psi_+(p)=\tanh(0)=0$. Since $\tanh$ is continuous and the linear functions $t\mapsto\gamma_0(p-t)$ and $t\mapsto\gamma_0(t-p)$ are continuous, $\psi_-$ is continuous on $(-\infty,p]$ and $\psi_+$ is continuous on $[p,\infty)$. At $x=p$,
\[
\lim_{x\to p^-}\psi_{p,\gamma_0}(x)=\lim_{x\to p^-}\psi_-(x)=\psi_-(p)=0,\qquad
\lim_{x\to p^+}\psi_{p,\gamma_0}(x)=\lim_{x\to p^+}\psi_+(x)=\psi_+(p)=0,
\]
and $\psi_{p,\gamma_0}(p)=\tanh(0)=0$, so $\psi_{p,\gamma_0}$ is continuous at $x=p$.

For $x<p$, by the chain rule,
\[
\psi_-'(x)=\operatorname{sech}^2\bigl(\gamma_0(p-x)\bigr)\cdot\gamma_0\cdot(-1),
\]
hence
\[
|\psi_-'(x)|=\gamma_0\operatorname{sech}^2\bigl(\gamma_0(p-x)\bigr)\le\gamma_0\qquad(\forall x<p),
\]
where the last inequality uses \eqref{eq:sech_bound}.

For $x>p$, similarly,
\[
\psi_+'(x)=\operatorname{sech}^2\bigl(\gamma_0(x-p)\bigr)\cdot\gamma_0,
\]
and thus
\[
|\psi_+'(x)|=\gamma_0\operatorname{sech}^2\bigl(\gamma_0(x-p)\bigr)\le\gamma_0\qquad(\forall x>p).
\]
At $x=p$, the right derivative is
\begin{align}\label{eq:right_deriv}
	\psi_+'(p^+)&=\lim_{h\to0^+}\frac{\psi_{p,\gamma_0}(p+h)-\psi_{p,\gamma_0}(p)}{h} \nonumber\\
	&=\lim_{h\to0^+}\frac{\tanh(\gamma_0 h)}{h}
	=\gamma_0\lim_{h\to0^+}\frac{\tanh(\gamma_0 h)}{\gamma_0\,h}
	=\gamma_0,
\end{align}
where we used 
\[
\displaystyle\lim_{z\to0}\frac{\tanh z}{z}=1.
\] 
The left derivative is
\[
\psi_-'(p^-)=\lim_{h\to0^+}\frac{\psi_{p,\gamma_0}(p-h)-\psi_{p,\gamma_0}(p)}{-h}
=\lim_{h\to0^+}\frac{\tanh\bigl(\gamma_0 h\bigr)}{-h}
= -\gamma_0\,\lim_{h\to0^+}\frac{\tanh\bigl(\gamma_0 h\bigr)}{\gamma_0\,h}
= -\gamma_0.
\]
Thus, $\psi_{p,\gamma_0}$ is not differentiable at $x=p$, but the absolute values of the left and right derivatives are $\gamma_0$.

We now prove the Lipschitz inequality. Take arbitrary $x,y\in\mathbb{R}$ and consider two cases.

\noindent\textbf{Case 1: $x$ and $y$ lie on the same side of $p$ (i.e., $x\le y\le p$ or $p\le x\le y$).}
If $x=y$, the inequality 
\[
|\psi_{p,\gamma_0}(x)-\psi_{p,\gamma_0}(y)|\le\gamma_0|x-y|
\]
holds trivially. Assume $x<y$. Then $\psi_{p,\gamma_0}$ is differentiable on the interval $[x,y]$ (using $\psi_-$ if $x,y\le p$, or $\psi_+$ if $x,y\ge p$). By the Mean Value Theorem, there exists $\xi\in(x,y)$ such that
\[
\psi_{p,\gamma_0}(y)-\psi_{p,\gamma_0}(x)=\psi_{p,\gamma_0}'(\xi)(y-x),
\]
where $\psi_{p,\gamma_0}'(\xi)=\psi_-'(\xi)$ or $\psi_+'(\xi)$. Consequently,
\[
|\psi_{p,\gamma_0}(y)-\psi_{p,\gamma_0}(x)|=|\psi_{p,\gamma_0}'(\xi)|\cdot|y-x|\le \gamma_0\,|y-x|.
\]
If $x>y$, exchanging $x$ and $y$ yields the same estimate. Hence, the inequality holds for all $x,y$ on the same side of $p$:
\begin{equation}\label{eq:lipschitz_same_side}
|\psi_{p,\gamma_0}(x)-\psi_{p,\gamma_0}(y)|\le\gamma_0|x-y|.
\end{equation}

\noindent\textbf{Case 2: $x$ and $y$ lie on opposite sides of $p$.}
Then the interval $[x,y]$ contains the point $p$ where $\psi_{p,\gamma_0}$ is not differentiable. We consider two subcases.

\noindent\textit{Subcase 2.1: $x\le p\le y$.}
Split the difference into two parts:
\[
|\psi_{p,\gamma_0}(y)-\psi_{p,\gamma_0}(x)|
\le
|\psi_{p,\gamma_0}(y)-\psi_{p,\gamma_0}(p)|
+
|\psi_{p,\gamma_0}(p)-\psi_{p,\gamma_0}(x)|.
\]
For $|\psi_{p,\gamma_0}(y)-\psi_{p,\gamma_0}(p)|$: Since $p\le y$, on the interval $[p,y]$ the function $\psi_{p,\gamma_0}$ coincides with $\psi_+$. By the Mean Value Theorem, there exists $\xi_1\in(p,y)$, such that
\[
\psi_+(y)-\psi_+(p)=\psi_+'(\xi_1)(y-p).
\]
Since $|\psi_+'(\xi_1)|\le\gamma_0$, and $\psi_+(y)=\psi_{p,\gamma_0}(y)$, $\psi_+(p)=\psi_{p,\gamma_0}(p)$, we obtain
\[
|\psi_{p,\gamma_0}(y)-\psi_{p,\gamma_0}(p)|=|\psi_+(y)-\psi_+(p)|\le\gamma_0|y-p|.
\]
For $|\psi_{p,\gamma_0}(p)-\psi_{p,\gamma_0}(x)|$: Since $x\le p$, on the interval $[x,p]$ the function $\psi_{p,\gamma_0}$ coincides with $\psi_-$. By the Mean Value Theorem, there exists $\xi_2\in(x,p)$, such that
\[
\psi_-(p)-\psi_-(x)=\psi_-'(\xi_2)(p-x).
\]
Again $|\psi_-'(\xi_2)|\le\gamma_0$, and $\psi_-(p)=\psi_{p,\gamma_0}(p)$, $\psi_-(x)=\psi_{p,\gamma_0}(x)$, hence
\[
|\psi_{p,\gamma_0}(p)-\psi_{p,\gamma_0}(x)|=|\psi_-(p)-\psi_-(x)|\le\gamma_0|p-x|.
\]
Combining these estimates yields
\begin{equation}\label{eq:opp_side_est1}
|\psi_{p,\gamma_0}(y)-\psi_{p,\gamma_0}(x)|
\le\gamma_0(|y-p|+|p-x|)=\gamma_0|y-x|.
\end{equation}

\noindent\textit{Subcase 2.2: $y\le p\le x$.}
Similarly, split the difference:
\[
|\psi_{p,\gamma_0}(x)-\psi_{p,\gamma_0}(y)|
\le|\psi_{p,\gamma_0}(x)-\psi_{p,\gamma_0}(p)|+|\psi_{p,\gamma_0}(p)-\psi_{p,\gamma_0}(y)|.
\]
For $|\psi_{p,\gamma_0}(x)-\psi_{p,\gamma_0}(p)|$: Since $p\le x$, on $[p,x]$ we use $\psi_+$, obtaining
\[
|\psi_{p,\gamma_0}(x)-\psi_{p,\gamma_0}(p)|\le\gamma_0|x-p|.
\]
For $|\psi_{p,\gamma_0}(p)-\psi_{p,\gamma_0}(y)|$: Since $y\le p$, on $[y,p]$ we use $\psi_-$, obtaining
\[
|\psi_{p,\gamma_0}(p)-\psi_{p,\gamma_0}(y)|\le\gamma_0|p-y|.
\]
Hence,
\begin{equation}\label{eq:opp_side_est2}
|\psi_{p,\gamma_0}(x)-\psi_{p,\gamma_0}(y)|
\le\gamma_0(|x-p|+|p-y|)=\gamma_0|x-y|.
\end{equation}

From Cases 1 and 2, we conclude that for all $x,y\in\mathbb{R}$,
\begin{equation}\label{eq:psi_lipschitz}
|\psi_{p,\gamma_0}(x)-\psi_{p,\gamma_0}(y)|\le\gamma_0|x-y|.
\end{equation}
Returning to the original notation, this is precisely
\[
\bigl|\tanh(\gamma_0|x-p|)-\tanh(\gamma_0|y-p|)\bigr|\le\gamma_0|x-y|\qquad(\forall x,y\in\mathbb{R}),
\]
so $\psi_{p,\gamma_0}$ is $\gamma_0$-Lipschitz continuous.
\end{proof}

Now we proceed to estimate $G_f(s)$. For clarity, introduce the temporary notation
\begin{align*}
	A_f(s) &:=
	\tanh\bigl( \gamma_0 \, |f(s) - r(s)| \bigr),\\
	B_f(s) &:=
	1 + \beta_0 \int_0^s e^{-\alpha_0(s-\tau)} \,
	\tanh\bigl( \gamma_0 \, |f(\tau) - r(\tau)| \bigr) \, d\tau,
\end{align*}
so that $G_f(s)=A_f(s)/B_f(s)$. Analogous quantities $A_g(s)$ and $B_g(s)$ are defined with $g$ in place of $f$.

From (P1), we already know that $0\le A_f(s)<1$ and $0\le A_g(s)<1$. Moreover, as established in~\eqref{eq:denom_lower_bound}, the positivity of the parameters ($\beta_0\ge0$, $e^{-\alpha_0(s-\tau)}\ge0$) together with $\tanh(\cdot)\ge0$ implies $B_f(s)\ge1$ and $B_g(s)\ge1$; in particular, $B_f(s)$ and $B_g(s)$ are strictly positive.

Consider the difference $|G_f(s)-G_g(s)|$. By a common denominator,
\begin{align}
	|G_f(s)-G_g(s)|
	&= \left| \frac{A_f(s)}{B_f(s)} - \frac{A_g(s)}{B_g(s)} \right| \nonumber\\[4pt]
	&= \left| \frac{A_f(s)\,B_g(s) - A_g(s)\,B_f(s)}{B_f(s)\,B_g(s)} \right| \nonumber\\[4pt]
	&= \frac{\left| A_f(s)\,B_g(s) - A_g(s)\,B_f(s) \right|}{B_f(s)\,B_g(s)}. \label{eq:G_diff_common}
\end{align}
Since $B_f(s)$ and $B_g(s)$ are positive, the denominator $B_f(s)B_g(s)$ is positive and therefore equals its absolute value.

To estimate the numerator, decompose it algebraically:
\begin{align}
	A_f(s)B_g(s)-A_g(s)B_f(s)
	&= A_f(s)B_g(s)-A_f(s)B_f(s)+A_f(s)B_f(s)-A_g(s)B_f(s) \nonumber\\
	&= A_f(s)\,[B_g(s)-B_f(s)] + B_f(s)\,[A_f(s)-A_g(s)]. \label{eq:G_diff_decomp}
\end{align}
Substituting \eqref{eq:G_diff_decomp} into \eqref{eq:G_diff_common} and applying the triangle inequality yields
\begin{equation}
	|G_f(s)-G_g(s)|
	\le
	\frac{A_f(s)\,|B_g(s)-B_f(s)|}{B_f(s)\,B_g(s)}
	+
	\frac{B_f(s)\,|A_f(s)-A_g(s)|}{B_f(s)\,B_g(s)}. \label{eq:G_diff_estimate}
\end{equation}

We now estimate the two terms on the right-hand side separately.

\noindent\textbf{Estimate of the first term.}
Since $A_f(s)\in[0,1)$ by (P1), we have $A_f(s)\le1$. Together with $B_f(s),B_g(s)\ge1$ from \eqref{eq:denom_lower_bound}, we obtain
\[
\frac{A_f(s)\,|B_g(s)-B_f(s)|}{B_f(s)\,B_g(s)} \le |B_g(s)-B_f(s)|.
\]
Now estimate $|B_g(s)-B_f(s)|$. From the definition,
\begin{align}
	|B_g(s)-B_f(s)|
	&= \Bigl| \ 
	\beta_0 \int_0^s e^{-\alpha_0(s-\tau)} \,
	\tanh\bigl( \gamma_0 |g(\tau)-r(\tau)| \bigr) \, d\tau \nonumber - \beta_0 \int_0^s e^{-\alpha_0(s-\tau)} \,
	\tanh\bigl( \gamma_0 |f(\tau)-r(\tau)| \bigr) \, d\tau
	\Bigr| \nonumber\\
	&= \beta_0 \Bigl|
	\int_0^s e^{-\alpha_0(s-\tau)} \,
	\Bigl[
	\tanh\bigl( \gamma_0 |g(\tau)-r(\tau)| \bigr)
	-
	\tanh\bigl( \gamma_0 |f(\tau)-r(\tau)| \bigr)
	\Bigr] \, d\tau
	\Bigr|. \label{eq:B_diff}
\end{align}
Applying the triangle inequality for integrals and using $e^{-\alpha_0(s-\tau)}\ge0$, we obtain
\[
|B_g(s)-B_f(s)|
\le
\beta_0 \int_0^s e^{-\alpha_0(s-\tau)} \,
\Bigl|
\tanh\bigl( \gamma_0 |g(\tau)-r(\tau)| \bigr)
-
\tanh\bigl( \gamma_0 |f(\tau)-r(\tau)| \bigr)
\Bigr| \, d\tau.
\]
For the integrand difference, apply Lemma~\ref{lem:tanh_abs_lipschitz} with $p=r(\tau)$. This yields
\[
\bigl|
\tanh\bigl( \gamma_0 \, |g(\tau)-r(\tau)| \bigr)
-
\tanh\bigl( \gamma_0 \, |f(\tau)-r(\tau)| \bigr)
\bigr|
\le
\gamma_0 \, |g(\tau)-f(\tau)|
=
\gamma_0 \, |f(\tau)-g(\tau)|.
\]
Substituting this estimate,
\begin{align}
	|B_g(s)-B_f(s)|
	&\le \beta_0 \int_0^s e^{-\alpha_0(s-\tau)} \,
	\gamma_0 \, |f(\tau)-g(\tau)| \, d\tau \nonumber\\
	&= \beta_0 \gamma_0 \int_0^s e^{-\alpha_0(s-\tau)} \,
	|f(\tau)-g(\tau)| \, d\tau \nonumber\\
	&\le \beta_0 \gamma_0 \|f-g\|_\infty
	\int_0^s e^{-\alpha_0(s-\tau)} \, d\tau, \label{eq:B_diff_estimate}
\end{align}
where $\|f-g\|_\infty:=\sup_{t\in I}|f(t)-g(t)|$. Computing the remaining integral,
\[
\int_0^s e^{-\alpha_0(s-\tau)} \, d\tau
=
\Bigl[ \frac{e^{-\alpha_0(s-\tau)}}{\alpha_0} \Bigr]_{\tau=0}^{\tau=s}
=
\frac{1-e^{-\alpha_0 s}}{\alpha_0}. \label{eq:exp_integral}
\]
Thus, we obtain the estimate for the first term:
\begin{equation}
	\frac{A_f(s)\,|B_g(s)-B_f(s)|}{B_f(s)\,B_g(s)}
	\le
	\beta_0 \gamma_0 \|f-g\|_\infty \cdot
	\frac{1-e^{-\alpha_0 s}}{\alpha_0}. \label{eq:first_term_estimate}
\end{equation}

\noindent\textbf{Estimate of the second term.}
From $B_f(s)\ge1$ and $B_g(s)\ge1$ established in \eqref{eq:denom_lower_bound}, we have
\[
\frac{B_f(s)}{B_f(s)\,B_g(s)}
=
\frac{1}{B_g(s)}
\le 1,
\]
and therefore
\[
\frac{B_f(s)\,|A_f(s)-A_g(s)|}{B_f(s)\,B_g(s)}
\le |A_f(s)-A_g(s)|.
\]
For $|A_f(s)-A_g(s)|$, apply Lemma~\ref{lem:tanh_abs_lipschitz} with $p=r(s)$. This yields
\[
|A_f(s)-A_g(s)|
=
\Bigl|
\tanh\bigl( \gamma_0 \, |f(s)-r(s)| \bigr)
-
\tanh\bigl( \gamma_0 \, |g(s)-r(s)| \bigr)
\Bigr|
\le
\gamma_0 \, |f(s)-g(s)|
\le
\gamma_0 \, \|f-g\|_\infty.
\]
Hence, the second term is bounded by
\begin{equation}
	\frac{B_f(s)\,|A_f(s)-A_g(s)|}{B_f(s)\,B_g(s)}
	\le \gamma_0 \, \|f-g\|_\infty. \label{eq:second_term_estimate}
\end{equation}

\noindent\textbf{Combining the two estimates.}
Substituting \eqref{eq:first_term_estimate} and \eqref{eq:second_term_estimate} into \eqref{eq:G_diff_estimate} yields
\begin{equation}
	|G_f(s)-G_g(s)|
	\le
	\beta_0 \gamma_0 \|f-g\|_\infty \cdot
	\frac{1-e^{-\alpha_0 s}}{\alpha_0}
	+
	\gamma_0 \, \|f-g\|_\infty
	=
	\gamma_0 \,
	\left(
	1 + \frac{\beta_0 (1-e^{-\alpha_0 s})}{\alpha_0}
	\right) \|f-g\|_\infty. \label{eq:G_final_estimate}
\end{equation}
Define the function
\[
L_G(s) := \gamma_0 \left(1 + \frac{\beta_0(1-e^{-\alpha_0 s})}{\alpha_0}\right).
\]
Then for every $s\in I$ and any two trajectories $f,g\in\mathcal{C}(I)$,
\[
|G_f(s)-G_g(s)| \le L_G(s)\,\|f-g\|_\infty.
\]

We now examine the behavior of $L_G(s)$ on the interval $I=[0,T]$. Since $\alpha_0>0$, $\beta_0\ge0$, and $\gamma_0>0$, the function $s\mapsto1-e^{-\alpha_0 s}$ is nonnegative and strictly increasing on $[0,\infty)$. Consequently, $L_G(s)$ is strictly increasing on $[0,\infty)$ (and constant only when $\beta_0=0$). In particular,
\begin{itemize}[leftmargin=*, nosep]
\item $L_G(0)=\gamma_0$;
\item $L_G(s)$ increases with $s$;
\item 
\[
\displaystyle\lim_{s\to\infty}L_G(s)=\gamma_0\left(1+\frac{\beta_0}{\alpha_0}\right).
\]
\end{itemize}

On the finite interval $I=[0,T]$, monotonicity implies that the maximum of $L_G$ is attained at the right endpoint:
\begin{equation}
	\max_{s\in[0,T]} L_G(s) = L_G(T) = \gamma_0 \left(1 + \frac{\beta_0 \, (1-e^{-\alpha_0 T})}{\alpha_0}\right). \label{eq:LG_max}
\end{equation}
Set
\[
L_G := \max_{s\in[0,T]} L_G(s) = \gamma_0 \left(1 + \frac{\beta_0 \, (1-e^{-\alpha_0 T})}{\alpha_0}\right) < \infty.
\]
Then for all $s\in I$ and any $f,g\in\mathcal{C}(I)$,
\[
|G_f(s)-G_g(s)|\le L_G\|f-g\|_\infty.
\]
Finally, recalling that $\Lambda_f(s)=\lambda_{\min}+(\lambda_{\max}-\lambda_{\min})G_f(s)$ and similarly for $\Lambda_g(s)$, we obtain for any $s\in I$,
\[
|\Lambda_f(s)-\Lambda_g(s)|
= (\lambda_{\max}-\lambda_{\min})\, |G_f(s)-G_g(s)|
\le (\lambda_{\max}-\lambda_{\min})\, L_G\, \|f-g\|_\infty.
\]
Taking the supremum over $s\in I$ on the left-hand side yields the desired estimate in the supremum norm:
\begin{equation}
\|\Lambda_f-\Lambda_g\|_\infty \le L_\Lambda\|f-g\|_\infty, \label{eq:Lambda_Lipschitz_sup}
\end{equation}
where
\[
L_\Lambda := (\lambda_{\max} - \lambda_{\min}) \, L_G
= (\lambda_{\max} - \lambda_{\min}) \, \gamma_0 \,
\left( 1 + \frac{\beta_0 \, (1 - e^{-\alpha_0 T})}{\alpha_0} \right).
\]

This completes the verification of (P2).

\noindent \textit{Proof of (P3): continuity (hence measurability) in time.}
Fix an arbitrary trajectory $f\in\mathcal{C}(I)$. To establish the continuity of $s\mapsto\Lambda_f(s)$ on $I$, we introduce the auxiliary functions
\begin{align*}
	\phi(s) &\coloneqq \tanh\bigl( \gamma_0 \, |f(s) - r(s)| \bigr), \\
	\Psi(s) &\coloneqq \int_0^s e^{-\alpha_0(s-\tau)} \,
	\tanh\bigl( \gamma_0 \, |f(\tau) - r(\tau)| \bigr) \, d\tau,
\end{align*}
so that 
\[
\Lambda_f(s)=\lambda_{\min}+(\lambda_{\max}-\lambda_{\min})\dfrac{\phi(s)}{1+\beta_0\Psi(s)}.
\]
First, $\phi$ is continuous on $I$. Indeed, $f$ and $r$ are continuous by hypothesis, the map $z\mapsto|z|$ is continuous, and the hyperbolic tangent is continuous; therefore, the composition
\[
\phi(s)=\tanh(\gamma_0|f(s)-r(s)|)
\] 
is continuous.

We now prove that $\Psi$ is continuous on $I$. Consider the function
\[
F(\tau,s):=e^{-\alpha_0(s-\tau)}\tanh\bigl(\gamma_0|f(\tau)-r(\tau)|\bigr)
\]
defined on the compact triangular region
\[
\triangle:=\{(\tau,s)\in I\times I\mid 0\le\tau\le s\le T\}.
\]
The exponential factor $e^{-\alpha_0(s-\tau)}$ is continuous on $\triangle$, and $\tanh(\gamma_0|f(\tau)-r(\tau)|)$ depends only on $\tau$ and is continuous. Hence, $F$ is continuous on the compact set $\triangle$, and therefore uniformly continuous on $\triangle$. Moreover, from $0\le\tanh(\cdot)<1$ and $0<e^{-\alpha_0(s-\tau)}\le1$, we obtain the uniform bound
\begin{equation}\label{eq:F_bound}
|F(\tau,s)|\le1\qquad\forall(\tau,s)\in\triangle.
\end{equation}

Take an arbitrary $s_0\in I$ and let $s\in I$. To estimate $|\Psi(s)-\Psi(s_0)|$, we consider two cases.

\noindent\textbf{Case 1: $s\ge s_0$.}
Then
\begin{align}
	|\Psi(s)-\Psi(s_0)|
	&=\Bigl|\int_0^s F(\tau,s)\,d\tau-\int_0^{s_0}F(\tau,s_0)\,d\tau\Bigr| \nonumber\\
	&\le \int_0^{s_0}|F(\tau,s)-F(\tau,s_0)|\,d\tau + \Bigl|\int_{s_0}^{s}F(\tau,s)\,d\tau\Bigr|. \label{eq:Psi_case1}
\end{align}
\noindent\textbf{Case 2: $s\le s_0$.}
Then
\begin{align}
	|\Psi(s)-\Psi(s_0)|
	&=\Bigl|\int_0^s F(\tau,s)\,d\tau-\int_0^{s_0}F(\tau,s_0)\,d\tau\Bigr| \nonumber\\
	&\le \int_0^{s}|F(\tau,s)-F(\tau,s_0)|\,d\tau + \Bigl|\int_{s}^{s_0}F(\tau,s_0)\,d\tau\Bigr|. \label{eq:Psi_case2}
\end{align}

Given any $\varepsilon>0$, by the uniform continuity of $F$ on $\triangle$ there exists $\delta_1>0$, such that whenever $|s-s_0|<\delta_1$ and $(\tau,s),(\tau,s_0)\in\triangle$,
\begin{equation}\label{eq:F_uniform}
|F(\tau,s)-F(\tau,s_0)|<\frac{\varepsilon}{2T}.
\end{equation}
For such $s$, the first term in either \eqref{eq:Psi_case1} or \eqref{eq:Psi_case2} (with the integration limit taken as $s_{\min}:=\min\{s,s_0\}$) satisfies
\[
\int_0^{s_{\min}}|F(\tau,s)-F(\tau,s_0)|\,d\tau
\le\frac{\varepsilon}{2T}\cdot T=\frac{\varepsilon}{2}.
\]
For the second term, we use \eqref{eq:F_bound}. If $|s-s_0|<{\varepsilon}/{2}$, then regardless of which case we are in,
\[
\Bigl|
\int_{s_{\min}}^{s_{\max}} F(\tau,\max\{s,\,s_0\})\,d\tau
\Bigr|
\le 1 \cdot |s_{\max}-s_{\min}| = |s-s_0| < \frac{\varepsilon}{2},
\]
where $s_{\max}:=\max\{s,s_0\}$ and the integrand $F(\tau,\max\{s,s_0\})$ is understood as $F(\tau,s)$ when $s\ge s_0$, and as $F(\tau,s_0)$ when $s\le s_0$.

Choosing $\delta \coloneqq \min\{\delta_1,\, \varepsilon/2\}$, we obtain $|\Psi(s)-\Psi(s_0)|<\varepsilon$ whenever $|s-s_0|<\delta$. Hence, $\Psi$ is continuous at $s_0$; since $s_0$ was arbitrary, $\Psi\in\mathcal{C}(I)$.

Because $\beta_0\ge0$ and $\Psi(s)\ge0$, the denominator $D(s):=1+\beta_0\Psi(s)$ satisfies $D(s)\ge1$ and is continuous. Consequently, 
\[
\Lambda_f(s)=\lambda_{\min}+(\lambda_{\max}-\lambda_{\min})\frac{\phi(s)}{D(s)}
\] 
is a combination of continuous functions by addition, multiplication and division by a non‑vanishing continuous denominator, and therefore, $\Lambda_f$ is continuous on $I$. Every continuous function is Lebesgue measurable, so (P3) is established.

\noindent \textit{Proof of (P4): positivity at zero.}
Consider the zero trajectory $f\equiv0$. Then $f(s)=0$ and $f(\tau)=0$ for all $s,\tau\in I$, and
\[
\Lambda_0(s)=\lambda_{\min}+(\lambda_{\max}-\lambda_{\min})\,
\frac{
	\tanh\bigl( \gamma_0 \, |r(s)| \bigr)
}{
	1 + \beta_0 \displaystyle\int_0^s e^{-\alpha_0(s-\tau)} \,
	\tanh\bigl( \gamma_0 \, |r(\tau)| \bigr) \, d\tau
}.
\]

We analyze the right-hand side step by step.

\begin{description}[leftmargin=*, nosep, font=\normalfont] 
    \item[(1)] For any $z\ge0$, $\tanh(z)\in[0,1)$; in particular, $\tanh(z)\ge0$.
    \item[(2)] Hence, $\tanh(\gamma_0|r(s)|)\ge0$ and $\tanh(\gamma_0|r(\tau)|)\ge0$ for all $s,\tau$.
    \item[(3)] The exponential $e^{-\alpha_0(s-\tau)}$ is always positive, and $\beta_0\ge0$; therefore, the integrand and the whole integral are non‑negative. Consequently,
    \begin{equation}\label{eq:denom_lower_bound_zero}
    1+\beta_0\int_0^s e^{-\alpha_0(s-\tau)}\tanh\bigl(\gamma_0|r(\tau)|\bigr)\,d\tau\;\ge\;1.
    \end{equation}
    \item[(4)] From the above, we obtain
    \[
    0 \le
    \frac{
    	\tanh\bigl( \gamma_0 \, |r(s)| \bigr)
    }{
    	1 + \beta_0 \displaystyle\int_0^s e^{-\alpha_0(s-\tau)} \,
    	\tanh\bigl( \gamma_0 \, |r(\tau)| \bigr) \, d\tau
    }
    < 1.
    \]
\end{description}

Since $0<\lambda_{\min}<\lambda_{\max}$, we have $\lambda_{\max}-\lambda_{\min}>0$, and therefore the fractional term multiplied by $(\lambda_{\max}-\lambda_{\min})$ is non‑negative. Thus,
\begin{equation}\label{eq:Lambda0_lower_bound}
\Lambda_0(s)\ge\lambda_{\min}\qquad\forall s\in I.
\end{equation}
The constant $\lambda_{\min}>0$ serves as the required positive lower bound, which is independent of $s$. This completes the verification of (P4).
\end{proof}

\begin{cor}[The purely instantaneous case ${\beta_0=0}$]
\label{cor:beta_zero_case}
Under the assumptions of Theorem~\ref{thm:lambda_f_properties}, suppose further that $\beta_0=0$. Then the construction in Example~\ref{ex:historical_deviation_adaptive} reduces to a function $\Lambda\in\mathscr{A}(I)$ given explicitly by
\begin{equation}\label{eq:Lambda_pure_instantaneous}
\Lambda(s,x) = \lambda_{\min} + (\lambda_{\max} - \lambda_{\min}) \tanh\bigl(\gamma_0 |x - r(s)|\bigr), \qquad (s,x)\in I\times\mathbb{R}.
\end{equation}
Consequently, for every trajectory $f\in\mathcal{C}(I)$, the associated function $\Lambda_f$ defined in \eqref{eq:lambda_f} satisfies
\[
\Lambda_f(s) = \Lambda(s, f(s)) \quad \text{for all } s\in I.
\]
We now verify that $\Lambda$ indeed belongs to the class $\mathscr{A}(I)$ by checking conditions (AS1)--(AS4) of Definition~\ref{def:adaptive_sensitivity}.

\noindent\textbf{Verification of (AS1): uniform boundedness.}
For any $(s,x)\in I\times\mathbb{R}$, the elementary bound $0\le\tanh(z)<1$ for $z\ge0$ yields
\[
0\le\tanh\bigl(\gamma_0|x-r(s)|\bigr)<1.
\]
Multiplying by $(\lambda_{\max}-\lambda_{\min})>0$ and adding $\lambda_{\min}$ gives
\begin{equation}\label{eq:cor_AS1}
\lambda_{\min} \le \Lambda(s,x) \le \lambda_{\min} + (\lambda_{\max}-\lambda_{\min}) = \lambda_{\max},
\end{equation}
which establishes (AS1) with the same constants $\lambda_{\min},\lambda_{\max}$.

\noindent\textbf{Verification of (AS2): Lipschitz continuity in the state variable.}
Fix $s\in I$ and set $p=r(s)$. Applying Lemma~\ref{lem:tanh_abs_lipschitz} with this $p$, we have for any $x,y\in\mathbb{R}$,
\[
\bigl|\tanh(\gamma_0|x-p|)-\tanh(\gamma_0|y-p|)\bigr|\le\gamma_0|x-y|.
\]
Consequently,
\begin{align}
|\Lambda(s,x)-\Lambda(s,y)|
&= (\lambda_{\max}-\lambda_{\min})\,
\bigl|\tanh(\gamma_0|x-r(s)|)-\tanh(\gamma_0|y-r(s)|)\bigr| \nonumber\\
&\le (\lambda_{\max}-\lambda_{\min})\gamma_0|x-y|. \label{eq:cor_AS2}
\end{align}
Thus, $\Lambda(s,\cdot)$ is Lipschitz continuous with constant $L_\Lambda = (\lambda_{\max}-\lambda_{\min})\gamma_0$, satisfying (AS2).

\noindent\textbf{Verification of (AS3): measurability in time.}
For each fixed $x\in\mathbb{R}$, consider the map $s\mapsto\Lambda(s,x)$. Since $r$ is continuous by hypothesis, the composition $s\mapsto\gamma_0|x-r(s)|$ is continuous; the hyperbolic tangent function is continuous, hence $s\mapsto\tanh(\gamma_0|x-r(s)|)$ is continuous. Adding the constant $\lambda_{\min}$ preserves continuity, so $s\mapsto\Lambda(s,x)$ is continuous on $I$. Every continuous function is Lebesgue measurable, therefore (AS3) holds.

\noindent\textbf{Verification of (AS4): positivity at zero.}
Setting $x=0$ in \eqref{eq:Lambda_pure_instantaneous} gives
\[
\Lambda(s,0) = \lambda_{\min} + (\lambda_{\max}-\lambda_{\min})\tanh\bigl(\gamma_0|r(s)|\bigr).
\]
Since $\tanh(\gamma_0|r(s)|)\ge0$ for all $s\in I$, we obtain
\begin{equation}\label{eq:cor_AS4}
\Lambda(s,0) \ge \lambda_{\min} > 0 \qquad \forall s\in I.
\end{equation}
Taking $\varrho := \lambda_{\min}$ (which is precisely the constant appearing in (AS1)), condition (AS4) is satisfied.

Having verified all four conditions (AS1)--(AS4), we conclude that $\Lambda$ belongs to the class $\mathscr{A}(I)$, thereby confirming that the purely instantaneous case $\beta_0=0$ is fully consistent with the framework of Definition~\ref{def:adaptive_sensitivity}. \qed
\end{cor}

\begin{rem}[Summary and perspective]
Theorem~\ref{thm:lambda_f_properties} establishes that for every $\beta_0\ge0$, the function $\Lambda_f$ generated by the historical deviation construction enjoys four fundamental properties. Properties (P1), (P3) and (P4) are direct analogues of the conditions (AS1), (AS3) and (AS4) that define the class $\mathscr{A}(I)$; they guarantee that each $\Lambda_f$ is uniformly bounded, continuous in time, and strictly positive at the zero trajectory. 

Property (P2), while playing a role analogous to (AS2) in the overall theory, takes a form adapted to the operator perspective: Instead of asserting pointwise Lipschitz dependence on the instantaneous state value $f(s)$, it provides a global estimate in the supremum norm, namely $\|\Lambda_f-\Lambda_g\|_\infty\le L_\Lambda\|f-g\|_\infty$. This formulation expresses the fact that the map $f\mapsto\Lambda_f$ is Lipschitz continuous from $\mathcal{C}(I)$ into itself---a natural and powerful statement when one views the construction as a nonlinear operator acting on entire trajectories. The constant $L_\Lambda$ captures the combined influence of parameters $\gamma_0$, $\beta_0$, and $\alpha_0$, and reduces to $(\lambda_{\max}-\lambda_{\min})\gamma_0$ in the purely instantaneous case $\beta_0=0$, consistently with Corollary~\ref{cor:beta_zero_case}.

Corollary~\ref{cor:beta_zero_case} examines the special case $\beta_0=0$, where the historical feedback vanishes. In this situation, the operator perspective reverts to the classical function viewpoint: There exists a fixed $\Lambda\in\mathscr{A}(I)$ such that $\Lambda_f(s)=\Lambda(s,f(s))$ for every trajectory $f$. This observation not only confirms that the construction reproduces the expected instantaneous sensitivity when no historical modulation is present, but also illustrates the flexibility of the framework; the same formula \eqref{eq:lambda_f} interpolates continuously between a standard $\mathscr{A}(I)$-function ($\beta_0=0$) and a genuinely history-dependent operator ($\beta_0>0$).

Together, Theorem~\ref{thm:lambda_f_properties} and Corollary~\ref{cor:beta_zero_case} demonstrate that the historical deviation example is both mathematically rich and perfectly compatible with the framework established in Definition~\ref{def:adaptive_sensitivity}. The operator viewpoint, with its Lipschitz estimate in the supremum norm, offers a natural language for describing memory mechanisms that depend on whole trajectories, while the special case $\beta_0=0$ seamlessly recovers the simpler instantaneous picture. This duality underscores the versatility of the proposed framework and its capacity to accommodate a wide spectrum of adaptive memory phenomena.

This framework thus provides a unified mathematical foundation for capturing nonlinear adaptive memory phenomena, including habituation, state-dependent weighting, and selective retention, that are central to understanding complex systems in fields such as neuroscience, adaptive control, and machine~learning.
\end{rem}

\section{Adaptive memory sets: functional construction and fundamental properties}
\label{sec:adaptive_memory_sets}

We finally arrive at the third layer of the framework: The adaptive memory-dependent functional and the associated function set, which integrate the kernel hierarchy and the adaptive sensitivity mechanism developed in the previous two sections. This layer provides a unified quantitative measure that combines instantaneous magnitude with adaptively weighted historical accumulation, thereby completing the three-layer structure.

The kernel hierarchy developed in Section~\ref{sec:kernel_theory} serves a twofold purpose. The regular class \(\mathscr{K}_{\mathrm{reg}}\) provides the analytical framework for the main results in Sections~\ref{sec:adaptive_sensitivity_framework} and \ref{sec:adaptive_memory_sets}. The generalized class \(\mathscr{K}_{\mathrm{gen}}\) is introduced as a broader theoretical interface, enabling sign-changing or discontinuous kernels that may arise in applications. The present analysis is developed primarily within the regular class, while the generalized class offers a potential starting point for further extensions under more general conditions.

The preceding sections have established two foundational pillars: A hierarchical classification of memory kernels $\kappa$ (Section~\ref{sec:kernel_theory}) that quantifies the temporal weighting of past states, and a framework for adaptive sensitivity functions $\Lambda$ (Section~\ref{sec:adaptive_sensitivity_framework}) that modulates this weight based on the state's value. In this section, we synthesize these components to construct a novel collection of functions tailored for systems exhibiting adaptive memory, a form of nonlinear behavior where the influence of past events depends not only on when they occurred but also on what their values were. This functional will serve as the fundamental analytical tool within this framework, thereby transforming the abstract theoretical setting into concrete functional-analytic machinery.

\subsection{Construction and well-posedness of fundamental memory functions}
\label{subsec:construction_functions}

To quantify the cumulative influence of a function's past under the adaptive memory paradigm, we first introduce two auxiliary functions. Their definitions rely on the synergy between the kernel $\kappa$ and the sensitivity function $\Lambda$, and their mathematical legitimacy must be rigorously established through a systematic verification before they can be employed in the definition of the adaptive memory-dependent functional.

\begin{defn}[Adaptive weighted cumulative function and instantaneous-memory hybrid function]
\label{def:adaptive_basic_functions}
Let $\kappa \in \mathscr{K}_{\mathrm{reg}}$ be a regular admissible kernel (Definition~\ref{def:regular_admissible_kernel}) and let $\Lambda \in \mathscr{A}(I)$ be an adaptive sensitivity function (Definition~\ref{def:adaptive_sensitivity}). For any continuous trajectory $f \in \mathcal{C}(I)$, where $\mathcal{C}(I)$ denotes the space of continuous functions on $I$ and is used throughout in this standard sense, we define the following:
\begin{description}[leftmargin=*, nosep, font=\normalfont] 
    \item[(i)] \textbf{Adaptive weighted cumulative function} $\mathcal{J}_f: I \to [0, \infty)$ by
    \begin{equation}
    \label{eq:J_f_definition}
    \mathcal{J}_f(t) := \int_0^t \Lambda\bigl(s, f(s)\bigr) \, \kappa(t-s) \, |f(s)| \, \mathrm{d}s, \qquad t \in I.
    \end{equation}
    This function quantifies the total weighted influence exerted on the present instant $t$ by all past states $f(s)$ for $0 \le s \le t$, accumulated through the adaptive memory weight $\Lambda(s,f(s))\kappa(t-s)$.
    
    \item[(ii)] \textbf{Instantaneous-memory hybrid function} $\mathcal{M}_f: I \to [0, \infty)$ by
    \begin{equation}
    \label{eq:M_f_definition}
    \mathcal{M}_f(t) := |f(t)| + \mathcal{J}_f(t), \qquad t \in I.
    \end{equation}
    This function simultaneously captures the system's instantaneous response $|f(t)|$ at time $t$ and the weighted historical accumulation $\mathcal{J}_f(t)$, thereby constituting the fundamental building block for the adaptive memory-dependent functional we seek to establish.
\end{description}
\end{defn}

\begin{rem}[Conceptual rationale behind the terminology]
\label{rmk:function_terminology}
The nomenclature in Definition~\ref{def:adaptive_basic_functions} is carefully chosen to reflect both the mathematical structure and the underlying physical intuition. The term adaptive weighted cumulative emphasizes that $\mathcal{J}_f$ is an integral (cumulative) measure, where the integrand is weighted by a factor that is simultaneously time-dependent (through $\kappa$) and state-dependent (through $\Lambda$). The designation instantaneous-memory hybrid for $\mathcal{M}_f$ accurately conveys its nature as a composite of a purely local quantity $|f(t)|$ and a global, history-dependent quantity $\mathcal{J}_f(t)$. This conceptual separation, distinguishing the accumulation mechanism and the combined instantaneous-historical measure, will prove invaluable in the subsequent analysis of the adaptive memory-dependent functional (see Definition~\ref{def:adaptive_memory_functional}) and related functional-analytic structures.

To illustrate the relevance of these constructions while respecting the standing regularity assumptions, one may interpret a trajectory $f(t)$ as, for instance, the firing rate of a neuron over time (or, more generally, a suitable continuous approximation thereof). The instantaneous term $|f(t)|$ captures the current level of neural activity, while the cumulative term $\mathcal{J}_f(t)$ aggregates past firing activity weighted by the elapsed time (via $\kappa$) and the firing rates (via $\Lambda$). The hybrid function $\mathcal{M}_f(t)$ therefore quantifies the combined influence of present and past activity, providing a mathematical description of neuronal adaptation, a nonlinear phenomenon where a neuron's response diminishes under sustained stimulation (habituation) or amplifies when exposed to novel or salient intermittent events (sensitization). This interpretation extends naturally to other domains: In ecological population dynamics, $f(t)$ may represent species abundance, with $\mathcal{M}_f(t)$ capturing the interplay between current population size and the historical constraints imposed by past population levels; in engineering signal processing, $f(t)$ may represent sensor readings, with $\mathcal{M}_f(t)$ encoding the balance between instantaneous measurements and historical trends. Such examples underscore how the abstract function $\mathcal{M}_f$ serves as a unified mathematical object for modeling state-dependent, history-sensitive adaptive behavior across scientific disciplines.
\end{rem}

Before these functions can be utilized as building blocks for the adaptive memory-dependent functional, it is imperative to verify that they are mathematically well-defined for the class of functions for which they are intended. The following lemma establishes their essential properties through a systematic argument.

\begin{lem}[Well-posedness of the fundamental memory functions]
\label{lem:functions_well_posed}
Let $f \in \mathcal{C}(I)$, $\kappa \in \mathscr{K}_{\mathrm{reg}}$, and $\Lambda \in \mathscr{A}(I)$. Then the functions introduced in Definition~\ref{def:adaptive_basic_functions} satisfy the following:
\begin{description}[leftmargin=*, nosep, font=\normalfont] 
    \item[(i)] \textbf{Absolute convergence and finiteness:} For every $t \in I$, the integral defining $\mathcal{J}_f(t)$ converges absolutely as a Lebesgue integral. Consequently, $\mathcal{J}_f(t)$ is a finite, non-negative real number for all $t\in I$, and $\mathcal{M}_f(t)$ is likewise finite and non-negative.

    \item[(ii)] \textbf{Measurability:} The mappings $t \mapsto \mathcal{J}_f(t)$ and $t \mapsto \mathcal{M}_f(t)$ are Lebesgue measurable on $I$.

    \item[(iii)] \textbf{Uniform boundedness:} There exists a finite constant $C>0$, independent of the specific trajectory $f$ and the temporal argument $t$, such that
    \begin{equation}
    \label{eq:M_f_uniform_bound}
    \mathcal{M}_f(t) \le C \|f\|_{\infty}, \qquad \forall t\in I.
    \end{equation}
    More explicitly, one may take $C := 1 + \Lambda_{\infty} \kappa_{\infty} T$, where
   \begin{align}
   	\Lambda_{\infty} &\coloneqq \sup_{(s,x)\in I\times\mathbb{R}} \left( \Lambda(s,x) \right) \le \lambda_{\mathrm{max}} < \infty, \label{eq:Lambda_sup_def}\\
   	\kappa_{\infty} &\coloneqq \sup_{\tau\in I} \kappa(\tau) \le M_\kappa < \infty, \label{eq:kappa_sup_def}
   \end{align}
    with $\lambda_{\mathrm{max}}$ and $M_\kappa$ being the uniform bounds from conditions (AS1) and (R1), respectively. As established in the proof above, one also obtains the auxiliary estimate $\mathcal{J}_f(t) \le \Lambda_{\infty} \kappa_{\infty} T \|f\|_{\infty}$ for all~$t\in I$.
\end{description}
\end{lem}

\begin{proof}
We establish each property through a systematic and self-contained argument.

\noindent \textit{Proof of (i): absolute convergence and finiteness.}
Since $I=[0,T]$ is compact and $f$ is continuous, $f$ is bounded on $I$. Denote
\begin{equation}
\label{eq:f_sup_norm}
M := \|f\|_{\infty} = \sup_{s\in I} |f(s)| < \infty.
\end{equation}
From the uniform boundedness condition (AS1) of Definition~\ref{def:adaptive_sensitivity}, there exists a constant $\lambda_{\mathrm{max}}>0$ such~that
\begin{equation}
\label{eq:Lambda_bound_pointwise}
|\Lambda(s, f(s))| \le \lambda_{\mathrm{max}}, \qquad \forall s\in I.
\end{equation}
Consider the dominating function for the integrand in \eqref{eq:J_f_definition}. For any $t\in I$ and $s\in[0,t]$, we have the pointwise estimate
\begin{equation}
	\label{eq:integrand_control}
	\left| \Lambda(s, f(s)) \, \kappa(t-s) \, \bigl| f(s) \bigr| \right|
	\le \lambda_{\mathrm{max}} M \cdot \kappa(t-s).
\end{equation}
The kernel $\kappa$ belongs to $\mathscr{K}_{\mathrm{reg}}$, hence by condition (R2) we have $\kappa\in L^1(I)$ and specifically 
\[
\int_0^T\kappa(\tau)\,\mathrm{d}\tau = 1.
\] 
Consequently, for each fixed $t$, the function $s\mapsto \kappa(t-s)$ is Lebesgue integrable on $[0,t]$. By the comparison test for Lebesgue integrals, the integral
\[
\int_0^t |\Lambda(s, f(s)) \kappa(t-s) |f(s)|| \, \mathrm{d}s
\]
converges, implying that $\mathcal{J}_f(t)$ is absolutely convergent and therefore finite. Non-negativity follows directly from $\kappa(\tau)\ge 0$ (condition (R1)) and $\Lambda(s,x)\ge 0$ (condition (AS1)). The finiteness and non-negativity of $\mathcal{M}_f(t)$ are then immediate from its definition \eqref{eq:M_f_definition}.

\noindent  \textit{Proof of (ii): measurability.} We establish the measurability of $\mathcal{J}_f$ through a product-measurability argument, which provides a rigorous foundation for applying Fubini-type theorems. Define the auxiliary function \(\mathcal{G} \colon I \times I \to [0, \infty)\) by
\begin{equation}
	\label{eq:G_function}
	\mathcal{G}(t,s) \coloneqq \Lambda(s, f(s)) \, \kappa(t-s) \, |f(s)| \, \mathbf{1}_{[0,t]}(s),
\end{equation}
where $\mathbf{1}_{[0,t]}$ denotes the indicator function of the interval $[0,t]$.

Let $\mathcal{L}(I)$ denote the $\sigma$-algebra of Lebesgue measurable sets on $I$, and let $\mathcal{L}(I)\otimes\mathcal{L}(I)$ be the corresponding product $\sigma$-algebra. We proceed by verifying the product measurability of each factor constituting $\mathcal{G}$.
\begin{description}[leftmargin=*, nosep, font=\normalfont]
	 
    \item[(a)] \textbf{Measurability of $(t,s)\mapsto \Lambda(s,f(s))$:}
    Consider the mapping $\Phi \colon I \to I\times\mathbb{R}$ defined by $\Phi(s) = (s, f(s))$. Since $f$ is continuous, $\Phi$ is continuous and therefore Borel measurable. The function 
   \[
    \Lambda: I\times\mathbb{R}\to[0,\infty)
   \] 
    is a Carathéodory function: By condition (AS2) in Definition~\ref{def:adaptive_sensitivity}, $x\mapsto\Lambda(s,x)$ is Lipschitz continuous (hence continuous) for each fixed $s$; by condition (AS3) in Definition~\ref{def:adaptive_sensitivity}, $s\mapsto\Lambda(s,x)$ is Lebesgue measurable for each fixed $x$. A Carathéodory function is jointly measurable with respect to the product $\sigma$-algebra $\mathcal{L}(I)\otimes\mathcal{B}(\mathbb{R})$, where $\mathcal{B}(\mathbb{R})$ denotes the Borel $\sigma$-algebra on $\mathbb{R}$. Consequently, the composition 
    \[
    H(s):=\Lambda(s,f(s)) = \Lambda\circ\Phi(s)
    \] 
    is Lebesgue measurable as the composition of a measurable function $\Lambda$ with a Borel measurable function $\Phi$. The extension $\tilde H(t,s):=H(s)$ is then product measurable, since for any $a\in\mathbb{R}$,
   \[
   \left\{ (t,s) \mid \tilde H(t,s) \le a \right\}
   =
   I \times \left\{ s \mid H(s) \le a \right\}
   \in \mathcal{L}(I) \otimes \mathcal{L}(I).
   \]
    
    \item[(b)] \textbf{Measurability of $(t,s)\mapsto \kappa(t-s)$:}
    Condition (R3) in Definition~\ref{def:regular_admissible_kernel} guarantees that $\kappa$ is Lipschitz continuous on $I$, hence continuous. The map $(t,s)\mapsto t-s$ is continuous, and the composition of continuous functions preserves continuity; therefore $(t,s)\mapsto\kappa(t-s)$ is continuous on $I\times I$, and consequently Borel measurable (hence $\mathcal{L}(I)\otimes\mathcal{L}(I)$-measurable).
    
    \item[(c)] \textbf{Measurability of $(t,s)\mapsto |f(s)|$:}
    The function $f$ is continuous, so $|f|$ is continuous. As in part (a), the extension $\tilde F(t,s):=|f(s)|$ is product measurable.
    
    \item[(d)] \textbf{Measurability of $(t,s)\mapsto \mathbf{1}_{[0,t]}(s)$:}
    Define the set 
   \[
    E:=\{(t,s)\in I\times I : 0\le s\le t\}.
    \] 
    The function $\phi(t,s)=t-s$ is continuous, hence $E=\phi^{-1}([0,\infty))$ is a closed set and therefore Borel measurable. Its characteristic function $\chi_E(t,s)=\mathbf{1}_{[0,t]}(s)$ is consequently Borel measurable and belongs to $\mathcal{L}(I)\otimes\mathcal{L}(I)$.
\end{description}

All four factors are non-negative and $\mathcal{L}(I)\otimes\mathcal{L}(I)$-measurable; hence, their product $\mathcal{G}$ is also $\mathcal{L}(I)\otimes\mathcal{L}(I)$-measurable and non-negative.

With $\mathcal{G}$ being non-negative and product-measurable, we can legitimately invoke the Fubini-Tonelli theorem for non-negative measurable functions. This classical result guarantees that:
\begin{itemize}[leftmargin=*, nosep]
    \item For almost every (in fact, every) $t\in I$, the section $s\mapsto\mathcal{G}(t,s)$ is $\mathcal{L}(I)$-measurable;
    \item The function defined by integrating over the second variable,
    \[
    t \mapsto \int_I \mathcal{G}(t,s)\,\mathrm{d}s
    = \int_0^t \Lambda(s,f(s))\,\kappa(t-s)\,|f(s)|\,\mathrm{d}s
    = \mathcal{J}_f(t),
    \]
    is $\mathcal{L}(I)$-measurable on $I$ (taking values in $[0,\infty)$).
\end{itemize}

Finally, the continuity of $f$ implies that $t\mapsto |f(t)|$ is continuous and therefore $\mathcal{L}(I)$-measurable. The~sum
\[
\mathcal{M}_f(t) = |f(t)| + \mathcal{J}_f(t)
\]
of two $\mathcal{L}(I)$-measurable functions remains $\mathcal{L}(I)$-measurable, completing the proof of part (ii).

\noindent \textit{Proof of (iii): uniform boundedness.} From the uniform boundedness condition (R1) for regular kernels, we have $\kappa(\tau)\le M_\kappa$ for all $\tau\in I$, where $M_\kappa$ is the constant appearing in Definition~\ref{def:regular_admissible_kernel}. Define the supremum norm of $\kappa$ on $I$ as
\begin{equation}
\label{eq:kappa_sup}
\kappa_{\infty} := \sup_{\tau\in I}\kappa(\tau) \le M_\kappa < \infty.
\end{equation}
Similarly, condition (AS1) for adaptive sensitivity functions (Definition~\ref{def:adaptive_sensitivity}) yields $\Lambda(s,x)\le \lambda_{\mathrm{max}}$ for all $(s,x)\in I\times\mathbb{R}$. Set
\begin{equation}
	\label{eq:Lambda_sup}
	\Lambda_{\infty} \coloneqq \sup_{(s,x)\in I\times\mathbb{R}} \Lambda(s,x) \le \lambda_{\mathrm{max}} < \infty.
\end{equation}
For any $t\in I$, using the non-negativity of $\kappa$ and the estimate $\kappa(t-s)\le \kappa_{\infty}$, we obtain
\[
\int_0^t \kappa(t-s)\,\mathrm{d}s \le \int_0^t \kappa_{\infty}\,\mathrm{d}s = \kappa_{\infty} t \le \kappa_{\infty} T.
\]
Combining these estimates, we bound the cumulative function $\mathcal{J}_f$ as follows:
\[
\mathcal{J}_f(t)
\le \int_0^t \Lambda_{\infty} \, \kappa(t-s) \, |f(s)| \,\mathrm{d}s
\le \Lambda_{\infty} \, \|f\|_{\infty} \int_0^t \kappa(t-s) \,\mathrm{d}s
\le \Lambda_{\infty} \, \|f\|_{\infty} \, \kappa_{\infty} \, T.
\]
Consequently, for the hybrid function $\mathcal{M}_f$, we have
\begin{align}
	\mathcal{M}_f(t) &= |f(t)| + \mathcal{J}_f(t) \nonumber\\
	&\le \|f\|_{\infty} + \Lambda_{\infty} \, \kappa_{\infty} \, T \, \|f\|_{\infty} \nonumber\\
	&= \bigl( 1 + \Lambda_{\infty} \, \kappa_{\infty} \, T \bigr) \|f\|_{\infty}. \label{eq:M_f_bound}
\end{align}
Taking $C := 1 + \Lambda_{\infty} \kappa_{\infty} T$, which is finite by virtue of \eqref{eq:kappa_sup} and \eqref{eq:Lambda_sup}, establishes the desired uniform estimate \eqref{eq:M_f_uniform_bound}. The bound for $\mathcal{J}_f$ follows directly from the chain of inequalities above.
\end{proof}

\begin{rem}[Significance of the well-posedness results]
\label{rmk:well_posed_significance}
Lemma~\ref{lem:functions_well_posed} provides a foundation for the subsequent theory. It shows that the auxiliary objects $\mathcal{J}_f$ and $\mathcal{M}_f$, which aim to capture the interplay between instantaneous magnitude, temporal weighting, and state-dependent sensitivity, are not merely formal constructions but objects with clear mathematical meaning. Absolute convergence ensures that no hidden cancellations or singularities affect the definition; measurability renders these functions amenable to integration and further functional-analytic operations, potentially enabling their incorporation into Lebesgue spaces; uniform boundedness offers an initial indication that $\mathcal{M}_f$ relates to the trajectory $f$ in a manner proportional to its supremum norm, with the proportionality constant $C$ depending only on the intrinsic parameters of the kernel and the sensitivity function. Collectively, these properties establish $\mathcal{M}_f$ as a mathematically viable object for quantifying adaptive memory. In particular, the state-dependent nature of $\Lambda$ within $\mathcal{J}_f$ introduces a nonlinear coupling between the trajectory $f$ and its own history, a feature that enables the resulting framework to capture adaptive phenomena such as habituation and selective retention.
\end{rem}

\subsection{Definition and basic properties of the adaptive memory-dependent functional}
\label{subsec:adaptive_functional_definition}

With the well-posedness of the fundamental functions $\mathcal{J}_f$ and $\mathcal{M}_f$ rigorously established in Lemma~\ref{lem:functions_well_posed}, we are now in a position to introduce the central object of our construction, the adaptive memory-dependent functional. Defined through a supremum operation, this functional provides a quantitative measure that integrates instantaneous magnitude and adaptively weighted historical accumulation. Its basic properties, which are essential for the subsequent functional-analytic developments, are systematically examined below.

\begin{defn}[Adaptive memory-dependent functional]
\label{def:adaptive_memory_functional}
Let $\kappa \in \mathscr{K}_{\mathrm{reg}}$ be a regular admissible kernel (Definition~\ref{def:regular_admissible_kernel}) and let $\Lambda \in \mathscr{A}(I)$ be an adaptive sensitivity function (Definition~\ref{def:adaptive_sensitivity}). For any continuous trajectory $f \in \mathcal{C}(I)$, the adaptive memory-dependent functional is defined by
\begin{equation}
\label{eq:adaptive_functional_definition}
S_{\kappa,\Lambda}(f) := \sup_{t \in I} \mathcal{M}_f(t) \in [0, \infty),
\end{equation}
where $\mathcal{M}_f$ is the instantaneous-memory hybrid function introduced in Definition~\ref{def:adaptive_basic_functions}. The uniform boundedness established in Lemma~\ref{lem:functions_well_posed} (iii) guarantees that this supremum is a finite nonnegative real~number.
\end{defn}

\begin{lem}[Basic properties of the adaptive memory-dependent functional]
\label{lem:adaptive_functional_properties}
Let $f \in \mathcal{C}(I)$. The functional $S_{\kappa,\Lambda}(f)$ defined in \eqref{eq:adaptive_functional_definition} satisfies the following fundamental properties:
\begin{description}[leftmargin=*, nosep, font=\normalfont] 
    \item[(i)] \textbf{Existence and well-posedness:} The quantity $S_{\kappa,\Lambda}(f)$ exists as a real number and is uniquely determined.
    
    \item[(ii)] \textbf{Controlled boundedness via intrinsic parameters:} 
    Recalling the uniform bounds 
    \[
    \Lambda_{\infty} := \sup_{(s,x)\in I\times\mathbb{R}} \Lambda(s,x) \quad \text{and} \quad 
    \kappa_{\infty} := \sup_{\tau\in I} \kappa(\tau)
    \]
    introduced in Lemma~\ref{lem:functions_well_posed} (see \eqref{eq:Lambda_sup_def} and \eqref{eq:kappa_sup_def}), we have the estimate
    \begin{equation}
    \label{eq:functional_controlled_bound}
    S_{\kappa,\Lambda}(f) \le \bigl(1 + \Lambda_{\infty} \kappa_{\infty} T\bigr) \|f\|_{\infty}.
    \end{equation}
    Conditions (AS1) and (R1) ensure $\Lambda_{\infty} \le \lambda_{\mathrm{max}} < \infty$ and $\kappa_{\infty} \le M_\kappa < \infty$, so the controlling constant is~finite.
    
    \item[(iii)] \textbf{Positive definiteness:} $S_{\kappa,\Lambda}(f) \ge 0$ for every $f\in\mathcal{C}(I)$. Moreover,
    \[
    S_{\kappa,\Lambda}(f) = 0 \quad \text{if and only if} \quad f \equiv 0 \text{ on } I.
    \]
    
       \item[(iv)] \textbf{Comparison with the classical supremum norm:} 
    The adaptive memory-dependent functional dominates the classical supremum norm, i.e.,
    \begin{equation}
    \label{eq:functional_comparison}
    \|f\|_{\infty} \le S_{\kappa,\Lambda}(f) \qquad \forall f\in\mathcal{C}(I).
    \end{equation}
\end{description}
\end{lem}

\begin{proof}
We establish each property through a systematic argument that leverages the results of Lemma~\ref{lem:functions_well_posed} while remaining self-contained.

\noindent \textit{Proof of (i): existence and well-posedness.}
From Lemma~\ref{lem:functions_well_posed} (i), the function $\mathcal{M}_f$ is well-defined pointwise and takes values in $[0,\infty)$. Specifically, for each $t\in I$, the quantity $\mathcal{M}_f(t)$ is a finite nonnegative real number. Hence the set
\[
S_f := \{\mathcal{M}_f(t) : t \in I\} \subset [0, \infty)
\]
is nonempty. Lemma~\ref{lem:functions_well_posed} (iii) provides the uniform bound
\begin{equation}
\label{eq:uniform_bound_for_S_f}
\mathcal{M}_f(t) \le (1 + \Lambda_{\infty} \kappa_{\infty} T) \|f\|_{\infty} \quad \text{for all } t\in I,
\end{equation}
which implies that $S_f$ is bounded above. By the completeness axiom of the real numbers, every nonempty subset of $\mathbb{R}$ that is bounded above possesses a unique supremum. Consequently, the quantity $\sup S_f$ exists as a real number and is uniquely determined. Defining $S_{\kappa,\Lambda}(f) := \sup S_f$ therefore yields a well-defined real number. Its non-negativity follows immediately from the fact that every element of $S_f$ is nonnegative, so the supremum of a set of nonnegative numbers is itself nonnegative.

\noindent \textit{Proof of (ii): controlled boundedness.}
Recall that $S_{\kappa,\Lambda}(f)$ is by definition the least upper bound of the set $S_f = \{\mathcal{M}_f(t) : t\in I\}$. A fundamental property of the supremum is that it cannot exceed any upper bound of the set. From Lemma~\ref{lem:functions_well_posed} (iii), we have the explicit upper bound
\[
\mathcal{M}_f(t) \le (1 + \Lambda_{\infty} \kappa_{\infty} T) \|f\|_{\infty} \quad \text{for every } t\in I.
\]
Thus the quantity $(1 + \Lambda_{\infty} \kappa_{\infty} T) \|f\|_{\infty}$ serves as an upper bound for the set $S_f$. Applying the aforementioned property of the supremum yields
\begin{equation}
	\label{eq:functional_bound_derivation}
	S_{\kappa,\Lambda}(f) = \sup_{t\in I} \mathcal{M}_f(t) \le \bigl( 1 + \Lambda_{\infty} \, \kappa_{\infty} \, T \bigr) \|f\|_{\infty},
\end{equation}
which is precisely the estimate \eqref{eq:functional_controlled_bound}.

\noindent \textit{Proof of (iii): positive definiteness.}
We establish the two directions of the equivalence separately.

\noindent \textbf{Non-negativity:} For any $f\in\mathcal{C}(I)$ and any $t\in I$, Lemma~\ref{lem:functions_well_posed} (i) guarantees that $\mathcal{M}_f(t) \ge 0$. The set $S_f = \{\mathcal{M}_f(t): t\in I\}$ therefore consists entirely of nonnegative numbers. The supremum of a set of nonnegative numbers is itself nonnegative, yielding $S_{\kappa,\Lambda}(f) = \sup S_f \ge 0$.

\noindent \textbf{Forward implication ($f\equiv 0 \Rightarrow S_{\kappa,\Lambda}(f) = 0$):} 
Assume that $f$ is identically zero on $I$, i.e., $f(t)=0$ for all $t\in I$. Then for any $t\in I$, we have $|f(t)| = 0$. Moreover, from the Definition \eqref{eq:J_f_definition} of $\mathcal{J}_f$, the integrand 
\[
\Lambda(s,0)\kappa(t-s)|0|
\]
vanishes for every $s\in[0,t]$, so $\mathcal{J}_f(t) = 0$. Consequently, $\mathcal{M}_f(t) = |f(t)| + \mathcal{J}_f(t) = 0$ for all $t\in I$. Thus, $S_f = \{0\}$, and its supremum is $\sup\{0\} = 0$. Hence, $S_{\kappa,\Lambda}(f) = 0$.

\noindent \textbf{Reverse implication ($S_{\kappa,\Lambda}(f) = 0 \Rightarrow f\equiv 0$):}
Suppose that $S_{\kappa,\Lambda}(f) = 0$. Since $S_{\kappa,\Lambda}(f)$ is defined as the supremum of the set $S_f = \{\mathcal{M}_f(t): t\in I\}$, and all elements of $S_f$ are nonnegative, the condition $\sup S_f = 0$ forces every element of $S_f$ to be zero. Indeed, if there existed some $t_0\in I$ with $\mathcal{M}_f(t_0) > 0$, then the supremum would be at least $\mathcal{M}_f(t_0) > 0$, contradicting $S_{\kappa,\Lambda}(f) = 0$. Therefore, $\mathcal{M}_f(t) = 0$ for every $t\in I$.

Now fix an arbitrary $t\in I$. From the definition $\mathcal{M}_f(t) = |f(t)| + \mathcal{J}_f(t)$ and the non-negativity of both terms, we have
\[
0 = \mathcal{M}_f(t) = |f(t)| + \mathcal{J}_f(t) \ge |f(t)| \ge 0.
\]
The chain of inequalities forces $|f(t)| = 0$, which implies $f(t) = 0$. Since $t\in I$ was chosen arbitrarily, we conclude that $f(t)=0$ for all $t\in I$, i.e., $f \equiv 0$ on $I$. This completes the proof of the reverse implication, and together with the forward implication establishes the equivalence
\begin{equation}
\label{eq:positive_definitiveness}
S_{\kappa,\Lambda}(f) = 0 \quad \text{if and only if} \quad f \equiv 0 \text{ on } I.
\end{equation}

\noindent \textit{Proof of (iv): comparison with the classical supremum norm.}
We first establish the pointwise inequality that underlies the comparison. For any $t\in I$, the definition of $\mathcal{M}_f$ from Definition~\ref{def:adaptive_basic_functions} gives
\[
\mathcal{M}_f(t) = |f(t)| + \mathcal{J}_f(t).
\]
Since $\mathcal{J}_f(t)$ is nonnegative (as established in Lemma~\ref{lem:functions_well_posed} (i)), we obtain the simple but crucial estimate
\begin{equation}
\label{eq:pointwise_comparison}
\mathcal{M}_f(t) \ge |f(t)| \quad \text{for every } t\in I.
\end{equation}
Taking the supremum over $t\in I$ on both sides of \eqref{eq:pointwise_comparison} preserves the inequality, yielding
\[
\sup_{t\in I} \mathcal{M}_f(t) \ge \sup_{t\in I} |f(t)|.
\]
By Definition~\ref{def:adaptive_memory_functional}, the left-hand side is $S_{\kappa,\Lambda}(f)$ and the right-hand side is the classical supremum norm $\|f\|_{\infty}$. Hence,
\[
S_{\kappa,\Lambda}(f) \ge \|f\|_{\infty} \quad \text{for all } f\in\mathcal{C}(I),
\]
which is precisely the inequality \eqref{eq:functional_comparison}.
\end{proof}

The inequality established in Lemma~\ref{lem:adaptive_functional_properties} (iv) is sharp in the sense that for non-zero functions it can be strengthened to a strict inequality, provided the supremum of $|f|$ is attained at a point where the subsequent estimate applies. The following theorem makes this precise.

\begin{thm}[Strict comparison when the maximizer lies in ${(0,T]}$]
\label{thm:strict_functional_comparison}
Let $f  \! \in  \! \mathcal{C}(I)$ be a non-zero function such that the maximum of $|f|$ is attained at some point $t^* \! \in \! (0,T]$, i.e., there exists $t^* \! \in \! (0,T]$~with
\begin{equation}
\label{eq:attainment_interior}
|f(t^*)| = \|f\|_{\infty} = \max_{t\in I} |f(t)|.
\end{equation}
Then the inequality established in Lemma~\ref{lem:adaptive_functional_properties} (iv) is strict, namely
\begin{equation}
\label{eq:strict_functional_comparison}
S_{\kappa,\Lambda}(f) > \|f\|_{\infty}.
\end{equation}
\end{thm}

\begin{proof}
Since $f$ is non-zero, we have $\|f\|_{\infty} > 0$, and consequently $|f(t^*)| > 0$, which implies $f(t^*) \neq 0$.

By the continuity of $f$ at $t^*$, for 
\[
\varepsilon = \dfrac{|f(t^*)|}{2} > 0,
\] 
there exists $\delta_0 > 0$ such that for all $s\in I$ with $|s - t^*| < \delta_0$, we have
\[
|f(s) - f(t^*)| < \frac{|f(t^*)|}{2}.
\]
Since $t^* > 0$ by hypothesis, we may choose $\delta$ sufficiently small such that $0 < \delta < \min\{\delta_0, t^*\}$. Then, for any $s \in (t^*-\delta, t^*)$ (note that $(t^*-\delta, t^*) \subset (t^*-\delta_0, t^*+\delta_0)$), the reverse triangle inequality yields
\begin{equation}
	\label{eq:f_lower_bound}
	|f(s)|
	\ge |f(t^*)| - |f(s) - f(t^*)|
	> |f(t^*)| - \frac{|f(t^*)|}{2}
	= \frac{|f(t^*)|}{2}
	> 0.
\end{equation}
Moreover, the interval $[t^*-\delta/2, t^*]$ is contained in $(t^*-\delta, t^*)$ and has positive length $\delta/2 > 0$. Consequently, the estimate $|f(s)| \ge |f(t^*)|/2$ holds for all $s \in [t^*-\delta/2, t^*]$.

We now estimate the integrand defining $\mathcal{J}_f(t^*)$ on the interval $[t^*-\delta/2, t^*]$. From condition (AS1) of Definition~\ref{def:adaptive_sensitivity}, we have the uniform lower bound $\Lambda(s,f(s)) \ge \lambda_{\min} > 0$ for all $s\in I$. For the kernel $\kappa$, condition (R1) of Definition~\ref{def:regular_admissible_kernel} provides the global lower bound $\kappa(\tau) \ge m_\kappa > 0$ for all $\tau\in I$. Since for any $s\in [t^*-\delta/2, t^*]$ we have $t^*-s \in [0, t^*] \subset I$, it follows that $\kappa(t^*-s) \ge m_\kappa > 0$.

Combining these estimates, we obtain for all $s$ in the interval $[t^*-\delta/2, t^*]$ the pointwise lower bound
\[
\Lambda(s,f(s))\kappa(t^*-s)|f(s)| \ge \lambda_{\min}\, m_\kappa\, \frac{|f(t^*)|}{2} > 0.
\]

Now we estimate $\mathcal{J}_f(t^*)$ from below. Restricting the integration domain to $[t^*-\delta/2, t^*]$ and using the pointwise lower bound, we obtain
\[
\mathcal{J}_f(t^*) \ge \int_{t^*-\delta/2}^{t^*} \lambda_{\min} \, m_\kappa \, \frac{|f(t^*)|}{2} \, ds = \lambda_{\min} \, m_\kappa \, \frac{|f(t^*)|}{2} \cdot \frac{\delta}{2} > 0.
\]
With this strict positivity established, we can now compare $\mathcal{M}_f(t^*)$ with $\|f\|_{\infty}$. Using \eqref{eq:attainment_interior} and the fact that $\mathcal{J}_f(t^*) > 0$, we obtain
\[
\mathcal{M}_f(t^*) = |f(t^*)| + \mathcal{J}_f(t^*) > |f(t^*)| = \|f\|_{\infty}.
\]
Finally, recalling from Definition~\ref{def:adaptive_memory_functional} that $S_{\kappa,\Lambda}(f)$ is the supremum of $\mathcal{M}_f(t)$ over $t\in I$, we have
\[
S_{\kappa,\Lambda}(f) = \sup_{t\in I} \mathcal{M}_f(t) \ge \mathcal{M}_f(t^*) > \|f\|_{\infty},
\]
which establishes the desired strict inequality \eqref{eq:strict_functional_comparison}. 
\end{proof}

\begin{rem}[Theoretical significance of the adaptive memory-dependent functional]
\label{rmk:functional_significance}
The functional $S_{\kappa,\Lambda}(f)$ introduced in Definition~\ref{def:adaptive_memory_functional} provides a quantitative framework for analyzing functions in the context of adaptive memory. Its construction incorporates two complementary mechanisms: Kernel $\kappa$ encodes the objective variation of memory influence with elapsed time, while the adaptive sensitivity function $\Lambda$ modulates this influence according to the state values $f(s)$ encountered along the trajectory. The composite weight $\Lambda(s,f(s))\kappa(t-s)$ thus captures a nuanced form of memory, wherein the weight depends on how much time has passed since the event and on the amplitude of the state at the time of the event.

The use of the supremum operation $\sup_{t\in I}$ in the definition ensures that the functional reflects the maximal combined impact that may occur over the time interval. The inequality
\[
\|f\|_{\infty} \le S_{\kappa,\Lambda}(f)
\]
 (Lemma~\ref{lem:adaptive_functional_properties} (iv)) shows that the adaptive memory-dependent functional is no weaker than the classical supremum norm; this fundamental comparison guarantees that the new functional does not underestimate the classical magnitude of a function, reflecting its basic compatibility, while laying the groundwork for further analysis of the conditions under which equality holds in this inequality and when it becomes strict. The basic properties established in Lemma~\ref{lem:adaptive_functional_properties}, existence, controlled boundedness, positive definiteness, and further confirm that $S_{\kappa,\Lambda}(\cdot)$ is a well-defined functional on $\mathcal{C}(I)$. Moreover, Theorem~\ref{thm:strict_functional_comparison} demonstrates that when the maximum of $|f|$ is attained in $(0,T]$, the adaptive memory-dependent functional is strictly greater than the classical supremum norm, reflecting the additional strictness contributed by the memory component.

The construction proceeds in a layered and sequential manner: First, the auxiliary functions $\mathcal{J}_f$ and $\mathcal{M}_f$ are introduced and their well-posedness established (Lemma~\ref{lem:functions_well_posed}); second, the functional is defined via a supremum operation (Definition~\ref{def:adaptive_memory_functional}); finally, its core properties are systematically verified (Lemma~\ref{lem:adaptive_functional_properties}) and a refined comparison result is obtained (Theorem~\ref{thm:strict_functional_comparison}).

Returning to the concrete contexts discussed in Remark~\ref{rmk:function_terminology}, the functional $S_{\kappa,\Lambda}(f)$ quantifies, for instance, the maximal combined impact of present and past neural activity in a single scalar value, thereby capturing nonlinear adaptive phenomena such as the reduced responsiveness after sustained stimulation (habituation) or the heightened sensitivity following a salient event. This unifying perspective illustrates how the abstract mathematical construction translates into a practical tool for analyzing state-dependent, history-sensitive adaptive behavior across applications.
\end{rem}

\subsection{Construction of the adaptive memory-dependent sensitivity set}
\label{subsec:adaptive_memory_set}

Having established the adaptive memory-dependent functional $S_{\kappa,\Lambda}(f)$ and its fundamental properties in the preceding subsection (see Definition~\ref{def:adaptive_memory_functional} and Lemma~\ref{lem:adaptive_functional_properties}), we now introduce the collection of functions that will serve as the primary setting for subsequent analysis. This collection is defined as the set of all functions on $I$ for which $S_{\kappa,\Lambda}(f)$ is finite.

\begin{rem}[Extension of the functional to broader function classes]
\label{rmk:functional_extension}
The adaptive memory-dependent functional $S_{\kappa,\Lambda}(f)$ was introduced in Definition~\ref{def:adaptive_memory_functional} for continuous functions $f \in \mathcal{C}(I)$. For a broader class of functions, the same expression
\[
S_{\kappa,\Lambda}(f) \coloneqq
\sup_{t \in I} \left(
|f(t)| + \int_0^t \Lambda(s, f(s)) \, \kappa(t-s) \, |f(s)| \, ds
\right)
\]
remains meaningful provided the integral exists as a Lebesgue integral. In particular, if $f$ is bounded and measurable, the composition $\Lambda(s, f(s))$ is measurable by the Carathéodory property of $\Lambda$, and the integral is well-defined. In the following, we consider functions for which $S_{\kappa,\Lambda}(f)$ is well-defined and~finite.
\end{rem}

\begin{defn}[Adaptive memory-dependent sensitivity set]
\label{def:adaptive_memory_set}
Let $\kappa \in \mathscr{K}_{\mathrm{reg}}$ be a regular admissible kernel (Definition~\ref{def:regular_admissible_kernel}) and let $\Lambda \in \mathscr{A}(I)$ be an adaptive sensitivity function (Definition~\ref{def:adaptive_sensitivity}). The adaptive memory-dependent sensitivity set is defined as
\begin{equation}
	\label{eq:adaptive_memory_set}
	\mathscr{M}_{\kappa,\Lambda}(I) \coloneqq
	\left\{
	f \colon I \to \mathbb{R}
	\;\middle|\;
	S_{\kappa,\Lambda}(f) < \infty
	\right\},
\end{equation}
where $S_{\kappa,\Lambda}(f)$ is understood in the sense of Remark~\ref{rmk:functional_extension}.
\end{defn}

\begin{rem}[On the set-theoretic characterization of $\mathscr{M}_{\kappa,\Lambda}(I)$]
\label{rmk:set_characterization}
Definition~\ref{def:adaptive_memory_set} does not impose any a priori regularity conditions on functions; membership is solely governed by the finiteness of $S_{\kappa,\Lambda}(f)$. As will be shown in Theorem~\ref{thm:embedding_C_into_M}, every continuous function on $I$ has finite $S_{\kappa,\Lambda}(f)$ value, so $\mathcal{C}(I) \subset \mathscr{M}_{\kappa,\Lambda}(I)$. Moreover, Proposition~\ref{prop:discontinuous_bounded_inclusion} demonstrates that the set also contains certain discontinuous functions, such as the indicator function of a subinterval. Consequently, the set $\mathscr{M}_{\kappa,\Lambda}(I)$ is strictly larger than $\mathcal{C}(I)$.
\end{rem}

\begin{thm}[Embedding of ${\mathcal{C}(I)}$ into the adaptive memory-dependent sensitivity set ${\mathscr{M}_{\kappa,\Lambda}(I)}$ and basic estimates]
\label{thm:embedding_C_into_M}
Let $\kappa \in \mathscr{K}_{\mathrm{reg}}$ and $\Lambda \in \mathscr{A}(I)$. Then the adaptive memory-dependent sensitivity set $\mathscr{M}_{\kappa,\Lambda}(I)$ satisfies the following properties:
\begin{description}[leftmargin=*, nosep, font=\normalfont] 
    \item[(1)] \textbf{Set-theoretic inclusion:} $\mathcal{C}(I) \subset \mathscr{M}_{\kappa,\Lambda}(I)$; in particular, $\mathscr{M}_{\kappa,\Lambda}(I)$ is nonempty.
    
    \item[(2)] \textbf{Comparison with the classical supremum norm and controlled boundedness:} For every $f \in \mathcal{C}(I)$,
    \begin{equation}
    \label{eq:functional_embedding_estimate}
    \|f\|_{\infty} \le S_{\kappa,\Lambda}(f) \le \bigl(1 + \Lambda_{\infty} \kappa_{\infty} T\bigr) \|f\|_{\infty},
    \end{equation}
    where 
    \[
    \Lambda_{\infty} := \sup_{(s,x)\in I\times\mathbb{R}} \Lambda(s,x) \quad\text{and}\quad \kappa_{\infty} := \sup_{\tau\in I} \kappa(\tau)
    \]
    are the uniform bounds introduced in Lemma~\ref{lem:functions_well_posed} (see \eqref{eq:Lambda_sup_def} and \eqref{eq:kappa_sup_def}).
\end{description}
\end{thm}

\begin{proof}
We establish each claim using the results of Lemma~\ref{lem:adaptive_functional_properties} and the uniform bounds $\Lambda_{\infty}$, $\kappa_{\infty}$ established therein.

\noindent \textit{Proof of (i): set-theoretic inclusion $\mathcal{C}(I) \subset \mathscr{M}_{\kappa,\Lambda}(I)$.}
Let $f \in \mathcal{C}(I)$ be arbitrary. Since $I = [0,T]$ is compact, $f$ is bounded; consequently, its supremum norm $\|f\|_{\infty} \!=\! \sup_{t\in I} |f(t)|$ is finite. By Lemma~\ref{lem:adaptive_functional_properties}~(ii), we have the estimate
\begin{equation}
\label{eq:S_finite_estimate}
S_{\kappa,\Lambda}(f) \le \bigl(1 + \Lambda_{\infty} \kappa_{\infty} T\bigr) \|f\|_{\infty}.
\end{equation}
The right-hand side is finite because:

\begin{itemize}[leftmargin=*, nosep]
    \item $\Lambda_{\infty} \le \lambda_{\mathrm{max}} < \infty$ by condition (AS1) of Definition~\ref{def:adaptive_sensitivity};
    \item $\kappa_{\infty} \le M_\kappa < \infty$ by condition (R1) of Definition~\ref{def:regular_admissible_kernel};
    \item $T < \infty$ is the fixed time horizon;
    \item $\|f\|_{\infty} < \infty$ as noted above.
\end{itemize}

Thus, $S_{\kappa,\Lambda}(f) < \infty$, and since $f$ is continuous, $S_{\kappa,\Lambda}(f)$ is well-defined. Hence, $f \in \mathscr{M}_{\kappa,\Lambda}(I)$. Since $f$ was arbitrary, we conclude $\mathcal{C}(I) \subset \mathscr{M}_{\kappa,\Lambda}(I)$.

Nonemptiness follows immediately: The zero function $\mathbf{0}(t) \equiv 0$ belongs to $\mathcal{C}(I)$, hence also to $\mathscr{M}_{\kappa,\Lambda}(I)$. Therefore $\mathscr{M}_{\kappa,\Lambda}(I) \neq \emptyset$.

\noindent \textit{Proof of (ii): comparison with the classical supremum norm and controlled boundedness.}
The left-hand inequality $\|f\|_{\infty} \le S_{\kappa,\Lambda}(f)$ is precisely Lemma~\ref{lem:adaptive_functional_properties} (iv), which holds for every $f \in \mathcal{C}(I)$. The right-hand inequality is exactly \eqref{eq:S_finite_estimate} derived above. Combining these yields the two-sided estimate~\eqref{eq:functional_embedding_estimate}.
\end{proof}

\begin{prop}[Controlled estimate and continuity of the inclusion mapping into ${\mathscr{M}_{\kappa,\Lambda}(I)}$]
\label{prop:inclusion_estimate}
Let $\kappa \in \mathscr{K}_{\mathrm{reg}}$ and $\Lambda \in \mathscr{A}(I)$. The identity mapping
\[
\iota: \mathcal{C}(I) \to \mathscr{M}_{\kappa,\Lambda}(I), \qquad \iota(f) = f,
\]
satisfies the following properties, where $\mathcal{C}(I)$ is equipped with the supremum norm $\|\cdot\|_{\infty}$ and $\mathscr{M}_{\kappa,\Lambda}(I)$ is equipped with the functional $S_{\kappa,\Lambda}$:
\begin{description}[leftmargin=*, nosep, font=\normalfont] 
    \item[(i)] \textbf{Linearity:} $\iota$ is a linear mapping.
    
    \item[(ii)] \textbf{Controlled estimate}: There exists a constant $C = 1 + \Lambda_{\infty} \kappa_{\infty} T$ such that
    \begin{equation}
    \label{eq:inclusion_boundedness}
    S_{\kappa,\Lambda}(\iota(f)) \le C \|f\|_{\infty} \quad \text{for all } f \in \mathcal{C}(I).
    \end{equation}
    
   \item[(iii)] \textbf{Continuity:} $\iota$ is uniformly continuous (in fact, Lipschitz) when $\mathcal{C}(I)$ is equipped with the supremum norm, since for any $f,g \in \mathcal{C}(I)$,
\begin{equation}
\label{eq:lipschitz_estimate}
S_{\kappa,\Lambda}(\iota(f) - \iota(g)) \le C \|f - g\|_{\infty}.
\end{equation}
\end{description}
\end{prop}

\begin{proof}
We establish each property in turn.

\noindent \textit{Proof of (i): linearity.}
For any $f, g \in \mathcal{C}(I)$ and any scalars $\alpha, \beta \in \mathbb{R}$,
\[
\iota(\alpha f + \beta g) = \alpha f + \beta g = \alpha \iota(f) + \beta \iota(g),
\]
since $\iota$ is defined as the identity mapping. Hence $\iota$ is linear. Note that this linearity follows directly from the definition of $\iota$ as the identity map and does not require $\mathscr{M}_{\kappa,\Lambda}(I)$ to be equipped with a vector-space structure.

\noindent \textit{Proof of (ii): controlled estimate.}
Let $C := 1 + \Lambda_{\infty} \kappa_{\infty} T$. For any $f \in \mathcal{C}(I)$, Theorem~\ref{thm:embedding_C_into_M} (ii) provides the estimate
\[
S_{\kappa,\Lambda}(f) \le C \|f\|_{\infty}.
\]
Since $S_{\kappa,\Lambda}(\iota(f)) = S_{\kappa,\Lambda}(f)$ by definition of $\iota$, the inequality \eqref{eq:inclusion_boundedness} follows immediately.

\noindent \textit{Proof of (iii): continuity.}
We first note that $\iota$ is linear by part (i). Consequently, for any $f, g \in \mathcal{C}(I)$,
\[
\iota(f) - \iota(g) = \iota(f - g).
\]
Now apply the controlled estimate \eqref{eq:inclusion_boundedness} to the function $f - g \in \mathcal{C}(I)$. This yields
\[
S_{\kappa,\Lambda}(\iota(f - g)) \le C \|f - g\|_{\infty}.
\]
Combining the two equalities, we obtain
\[
S_{\kappa,\Lambda}(\iota(f) - \iota(g)) = S_{\kappa,\Lambda}(\iota(f - g)) \le C \|f - g\|_{\infty},
\]
which is precisely \eqref{eq:lipschitz_estimate}. This inequality shows that $\iota$ is Lipschitz continuous with Lipschitz constant $C$ with respect to the distance induced by the supremum norm on $\mathcal{C}(I)$. In particular, for any $\varepsilon > 0$, choosing $\delta = \varepsilon / C$ gives
\[
S_{\kappa,\Lambda}(\iota(f) - \iota(g)) < \varepsilon \quad \text{whenever} \quad \|f - g\|_{\infty} < \delta,
\]
which is the definition of uniform continuity. Hence, $\iota$ is uniformly continuous.
\end{proof}

\begin{prop}[A bounded discontinuous function belongs to ${\mathscr{M}_{\kappa,\Lambda}(I)}$]
\label{prop:discontinuous_bounded_inclusion}
Let $\kappa \in \mathscr{K}_{\mathrm{reg}}$ and $\Lambda \in \mathscr{A}(I)$ be arbitrary. For any fixed $t_\star \in (0,T)$, define
\begin{equation}
	\label{eq:step_function}
	f_{t_\star}(t) \coloneqq \mathbf{1}_{[0,t_\star]}(t) = \begin{cases}
		1, & 0 \le t \le t_\star,\\[4pt]
		0, & t_\star < t \le T.
	\end{cases}
\end{equation}
Then $f_{t_\star}$ is discontinuous at $t = t_\star$ (hence $f_{t_\star} \notin \mathcal{C}(I)$), yet $f_{t_\star} \in \mathscr{M}_{\kappa,\Lambda}(I)$; i.e., $S_{\kappa,\Lambda}(f_{t_\star})$ is well-defined and finite.
\end{prop}

\begin{proof}
We first verify that $S_{\kappa,\Lambda}(f_{t_\star})$ is well-defined. The function $f_{t_\star}$ defined in \eqref{eq:step_function} is piecewise constant, hence Borel measurable. Since $\Lambda$ is a Carathéodory function (measurable in $s$ and continuous in $x$), the composition $\Lambda(s, f_{t_\star}(s))$ is Lebesgue measurable. Moreover, $f_{t_\star}$ is bounded by $1$, and $\kappa$ is bounded by $\kappa_{\infty} < \infty$. Consequently, the integral defining $\mathcal{J}_{f_{t_\star}}(t)$ exists as a Lebesgue integral for each $t \in I$, and $S_{\kappa,\Lambda}(f_{t_\star})$ is well-defined.

Now, by the definitions of $\mathcal{M}_f$ and $S_{\kappa,\Lambda}$ (see Definitions~\ref{def:adaptive_basic_functions} and~\ref{def:adaptive_memory_functional}, together with Remark~\ref{rmk:functional_extension} for the extension to broader function classes),for any $t \in I$, we have
\[
\mathcal{M}_{f_{t_\star}}(t) = |f_{t_\star}(t)| + \mathcal{J}_{f_{t_\star}}(t),
\]
where
\[
\mathcal{J}_{f_{t_\star}}(t) = \int_0^t \Lambda(s, f_{t_\star}(s)) \, \kappa(t-s) \, |f_{t_\star}(s)| \, ds.
\]
Observe that $|f_{t_\star}(s)| \le 1$ for all $s \in I$. By condition (AS1) of Definition~\ref{def:adaptive_sensitivity}, we have 
\[
\Lambda(s, f_{t_\star}(s)) \le \Lambda_{\infty} < \infty
\]
for all $s \in I$, where $\Lambda_{\infty} := \sup_{(s,x)\in I\times\mathbb{R}} \Lambda(s,x)$. By condition (R1) of Definition~\ref{def:regular_admissible_kernel}, we have 
\[
\kappa(t-s) \le \kappa_{\infty} < \infty
\] 
for all $t,s \in I$ with $s \le t$, where $\kappa_{\infty} := \sup_{\tau \in I} \kappa(\tau)$. Consequently,
\begin{equation}
	\label{eq:Jf_bound_discontinuous}
	\mathcal{J}_{f_{t_\star}}(t)
	\le \Lambda_{\infty} \, \kappa_{\infty} \int_0^t 1 \, ds
	= \Lambda_{\infty} \, \kappa_{\infty} \, t
	\le \Lambda_{\infty} \, \kappa_{\infty} \, T.
\end{equation}
Since $|f_{t_\star}(t)| \le 1$ as well, we obtain from \eqref{eq:Jf_bound_discontinuous} that for every $t \in I$,
\[
\mathcal{M}_{f_{t_\star}}(t) = |f_{t_\star}(t)| + \mathcal{J}_{f_{t_\star}}(t) \le 1 + \Lambda_{\infty} \kappa_{\infty} T.
\]
Therefore,
\begin{equation}
\label{eq:S_bound_discontinuous}
S_{\kappa,\Lambda}(f_{t_\star}) = \sup_{t \in I} \mathcal{M}_{f_{t_\star}}(t) \le 1 + \Lambda_{\infty} \kappa_{\infty} T < \infty,
\end{equation}
which proves $f_{t_\star} \in \mathscr{M}_{\kappa,\Lambda}(I)$ by Definition~\ref{def:adaptive_memory_set}.
\end{proof}

\begin{rem}[Mathematical significance of the adaptive memory-dependent sensitivity set and its structure]
\label{rmk:set_significance}
Theorem~\ref{thm:embedding_C_into_M} and Propositions~\ref{prop:inclusion_estimate} and \ref{prop:discontinuous_bounded_inclusion} collectively elucidate the nature of $\mathscr{M}_{\kappa,\Lambda}(I)$. Several aspects of this structure merit elaboration.

\noindent\textbf{Connection with continuous functions.}
For continuous functions, Lemma~\ref{lem:adaptive_functional_properties} (iv) yields the inequality $\|f\|_{\infty} \le S_{\kappa,\Lambda}(f)$, showing that the adaptive memory-dependent functional dominates the classical supremum norm. Theorem~\ref{thm:embedding_C_into_M} (ii) refines this observation into a two-sided estimate:
\[
\|f\|_{\infty} \le S_{\kappa,\Lambda}(f) \le (1 + \Lambda_{\infty} \kappa_{\infty} T) \|f\|_{\infty} \qquad \forall f \in \mathcal{C}(I).
\]
Thus, every continuous function belongs to $\mathscr{M}_{\kappa,\Lambda}(I)$, i.e., $\mathcal{C}(I) \subset \mathscr{M}_{\kappa,\Lambda}(I)$, and on this subspace the functional, $S_{\kappa,\Lambda}$, is equivalent to the classical supremum norm. The inclusion mapping $\iota: \mathcal{C}(I) \to \mathscr{M}_{\kappa,\Lambda}(I)$ (where $\mathcal{C}(I)$ is equipped with $\|\cdot\|_{\infty}$ and $\mathscr{M}_{\kappa,\Lambda}(I)$ is equipped with $S_{\kappa,\Lambda}$) is linear, bounded, and continuous; these properties are recorded in Proposition~\ref{prop:inclusion_estimate}.

\noindent\textbf{Enlargement beyond continuity.}
The set $\mathscr{M}_{\kappa,\Lambda}(I)$ is strictly larger than $\mathcal{C}(I)$. Proposition~\ref{prop:discontinuous_bounded_inclusion} provides an explicit illustration: The indicator function of a subinterval,
\[
f_{t_\star}(t) = \mathbf{1}_{[0,t_\star]}(t),
\]
is discontinuous at $t = t_\star$ (hence $f_{t_\star} \notin \mathcal{C}(I)$), yet satisfies $S_{\kappa,\Lambda}(f_{t_\star}) < \infty$, and therefore belongs to $\mathscr{M}_{\kappa,\Lambda}(I)$. This example shows that a function can belong to $\mathscr{M}_{\kappa,\Lambda}(I)$ without being continuous. This flexibility is particularly relevant for modeling nonlinear phenomena where signals exhibit jump discontinuities, such as abrupt environmental changes or on-off switching in biological and engineered~systems.

\noindent\textbf{Quantitative comparison.}
The functional $S_{\kappa,\Lambda}(f)$ merges the instantaneous amplitude $|f(t)|$ with the adaptively weighted historical accumulation $\mathcal{J}_f(t)$, thereby capturing the entire evolution of $f$. Theorem~\ref{thm:strict_functional_comparison} reveals a refined quantitative property: When the maximum of $|f|$ is attained in $(0,T]$, the adaptive memory-dependent functional strictly exceeds the classical supremum norm,
\[
S_{\kappa,\Lambda}(f) > \|f\|_{\infty},
\]
which reflects the influence of historical accumulation on the value of $S_{\kappa,\Lambda}(f)$.

\noindent\textbf{Methodological perspective.}
The construction of $\mathscr{M}_{\kappa,\Lambda}(I)$ follows a layered approach: Starting from the kernel $\kappa$ and the sensitivity function $\Lambda$, we first introduced the auxiliary functions $\mathcal{J}_f$ and $\mathcal{M}_f$ and established their well-posedness (Lemma~\ref{lem:functions_well_posed}); we then defined the functional $S_{\kappa,\Lambda}$ via a supremum operation (Definition~\ref{def:adaptive_memory_functional}) and verified its basic properties (Lemma~\ref{lem:adaptive_functional_properties}); finally, we introduced the set $\mathscr{M}_{\kappa,\Lambda}(I)$ and examined its relationship with $\mathcal{C}(I)$ (Theorem~\ref{thm:embedding_C_into_M}, Propositions~\ref{prop:inclusion_estimate} and \ref{prop:discontinuous_bounded_inclusion}). This structured development offers a possible framework for investigating functions and functionals within the context of adaptive memory.

\noindent\textbf{Nonlinear significance.}
The set $\mathscr{M}_{\kappa,\Lambda}(I)$ is defined through $S_{\kappa,\Lambda}(f)$, whose construction fundamentally relies on the state-dependent sensitivity factor $\Lambda(s,f(s))$. This factor introduces a nonlinear coupling between the trajectory $f$ and its own history: The weight assigned to a past event depends not only on when it occurred (via $\kappa$) but also on the value of $f$ at that time (via $\Lambda$). As illustrated in Remark~\ref{rmk:function_terminology}, this nonlinearity enables the description of adaptive phenomena such as habituation (diminished sensitivity under sustained stimulation) and selective retention (enhanced memory of salient events). Thus, $\mathscr{M}_{\kappa,\Lambda}(I)$ can be viewed as a mathematical construct designed for analyzing systems where memory is time-dependent and state-dependent.
\end{rem}

\section{Conclusions}

This work has been devoted to a functional-analytic framework for nonlinear adaptive memory, in which the influence of past events is governed by both elapsed time and the state values along a trajectory. This dual dependence is realized through a three-layer organization: A hierarchical classification of kernels, a class of state-dependent sensitivity functions, and a memory-dependent functional with an associated function set. The main contributions of each layer can be summarized as~follows.

First, we have introduced a hierarchical classification of memory kernels, consisting of mathematically admissible kernels (Definition~\ref{def:math_admissible_kernel}), regular admissible kernels (Definition~\ref{def:regular_admissible_kernel}), and generalized admissible kernels (Definition~\ref{def:general_admissible_kernel}). This classification was illustrated by three prototypical examples---the exponential kernel (Example~\ref{ex:exponential_kernel}), the power-law kernel (Example~\ref{ex:power_law_kernel_generalized}), and the finite-memory kernel (Example~\ref{ex:finite_memory_kernel})---which together illustrate short-memory, long-memory, and finite-window behaviors. The associated integral estimates (Lemma~\ref{lem:kernel_integral_estimates} and Proposition~\ref{prop:generalized_estimates}) provide the analytical tools required for the subsequent construction of the adaptive sensitivity framework in Section~\ref{sec:adaptive_sensitivity_framework}.

Second, we have introduced the class of adaptive sensitivity functions (Definition~\ref{def:adaptive_sensitivity}), which modulate the memory weight according to the state value \(f(s)\) along the trajectory. A concrete construction based on historical deviation accumulation (Example~\ref{ex:historical_deviation_adaptive}) was analyzed in detail, and its fundamental properties---uniform boundedness, a Lipschitz-type estimate with respect to the supremum norm, continuity in time, and positivity at the zero trajectory---were established in Theorem~\ref{thm:lambda_f_properties}. The special case of purely instantaneous sensitivity was obtained as a special case in Corollary~\ref{cor:beta_zero_case}, showing that the construction interpolates continuously between instantaneous response and genuinely history-dependent sensitivity.

Third, we have constructed the adaptive memory-dependent functional \(S_{\kappa,\Lambda}(f)\) (Definition~\ref{def:adaptive_memory_functional}) and the associated function set \(\mathscr{M}_{\kappa,\Lambda}(I)\) (Definition~\ref{def:adaptive_memory_set}). The well-posedness of the underlying auxiliary functions \(\mathcal{J}_f\) and \(\mathcal{M}_f\) was established in Lemma~\ref{lem:functions_well_posed}, and the fundamental properties of the functional---existence, controlled boundedness, positive definiteness, and comparison with the classical supremum norm---were verified in Lemma~\ref{lem:adaptive_functional_properties}. A strict inequality was proved in Theorem~\ref{thm:strict_functional_comparison} when the maximum of \(|f|\) is attained in \((0,T]\). The embedding \(\mathcal{C}(I) \subset \mathscr{M}_{\kappa,\Lambda}(I)\) and the two-sided estimate
\begin{equation}
\|f\|_{\infty} \le S_{\kappa,\Lambda}(f) \le \bigl(1 + \Lambda_{\infty} \kappa_{\infty} T\bigr) \|f\|_{\infty}
\label{eq:conclusions_two_sided}
\end{equation}
were established in Theorem~\ref{thm:embedding_C_into_M}, showing that \(S_{\kappa,\Lambda}\) and the classical supremum norm control each other on \(\mathcal{C}(I)\). The inclusion mapping was further shown to be linear, to satisfy a controlled estimate, and to be Lipschitz continuous (Proposition~\ref{prop:inclusion_estimate}), and the set \(\mathscr{M}_{\kappa,\Lambda}(I)\) was shown to contain bounded discontinuous functions such as the indicator function of a subinterval (Proposition~\ref{prop:discontinuous_bounded_inclusion}), thereby confirming that \(\mathscr{M}_{\kappa,\Lambda}(I)\) is strictly larger than \(\mathcal{C}(I)\).

Several directions remain open for future investigation. On the theoretical side, extending the main results from the regular kernel class \(\mathscr{K}_{\mathrm{reg}}\) to the generalized kernel class \(\mathscr{K}_{\mathrm{gen}}\) would broaden the applicability of the framework to sign-changing and discontinuous memory mechanisms, including inhibitory neural feedback, abrupt system resets, and adaptive gain control. On the applied side, incorporating the proposed functional into concrete models from fields such as neuroscience, adaptive control, and machine learning would help to assess its descriptive power and to explore its potential for practical implementation. We hope that the present work offers a constructive step toward a broader mathematical understanding of nonlinear adaptive memory phenomena.

\section*{Acknowledgments}

This work was supported by the Fundamental Research Funds for the Central Universities (a Chinese national program for basic research in higher education institutions) under Grant No. 2682026CX035. The author would like to express sincere gratitude to the Editors, the Assistant Editor, and the anonymous reviewers for their careful reading of the manuscript and for their insightful, constructive comments, which have greatly improved the clarity and rigor of this work. The author is also grateful for the consistently positive and encouraging feedback received from all reviewers during the review process, and for the efficient handling and thoughtful guidance provided by the editorial team throughout the submission.

\section*{Use of Generative-AI tools declaration}
The author declares that no Generative-AI (Artificial Intelligence) tools were used in the creation of this article.

\section*{Conflict of interest}
The author declares that there is no conflict of interest.

\end{document}